\documentclass[12pt]{amsart}

\usepackage{amsmath,amsthm,amssymb,mathrsfs}

\usepackage{mathptmx,amsmath,amssymb,mathrsfs}
\usepackage[mathcal]{eucal}

\newtheorem{theorem}{Theorem}[section]
\newtheorem{lemma}[theorem]{Lemma}
\newtheorem{corollary}[theorem]{Corollary}
\newtheorem{remark}[theorem]{Remark}

\numberwithin{equation}{section}

\newenvironment{mlist}{\list{}{\listparindent 0pt
\itemsep 0pt \parsep 2pt \topsep 1pt
}}{\endlist}

\newcommand{\sF}{\mathscr{F}}
\newcommand{\sH}{\mathscr{H}}
\newcommand{\sK}{\mathscr{K}}

\newcommand{\cK}{\mathcal{K}}
\newcommand{\cR}{\mathcal{R}}

\newcommand{\ad}{\operatorname{ad}}
\newcommand{\Ad}{\operatorname{Ad}}
\newcommand{\tr}{\operatorname{tr}}

\newcommand{\card}{\operatorname{card}}
\newcommand{\diag}{\operatorname{diag}}
\newcommand{\End}{\operatorname{End}}
\newcommand{\Id}{{\operatorname{Id}\kern.4pt}}
\newcommand{\rank}{\operatorname{rank}}

\newcommand{\eqdef}{\stackrel{\mathrm{def}}{=}}
\newcommand{\tith}{{\widetilde\theta}}
\newcommand{\tio}{{\widetilde\omega}}
\newcommand{\tw}{{*}}

\newcommand{\fa}{\mathfrak{a}}
\newcommand{\fg}{\mathfrak{g}}
\newcommand{\fh}{\mathfrak{h}}
\newcommand{\fk}{\mathfrak{k}}
\newcommand{\fm}{\mathfrak{m}}
\newcommand{\ft}{\mathfrak{t}}
\newcommand{\fz}{\mathfrak{z}}

\newcommand{\rd}{\mathrm{d}}

\newcommand{\ri}{\mathrm{i}}
\newcommand{\ai}{{i}}

\newcommand{\bbR}{\mathbb{R}}
\newcommand{\bbZ}{\mathbb{Z}}
\newcommand{\bbC}{\mathbb{C}}

\newcommand{\bfg}{\mathbf{g}}
\newcommand{\bfr}{\mathbf{r}}

\begin{document}

\title[Ricci-flat K\"ahler metrics on tangent bundles of
rank-one symmetric spaces]{Ricci-flat K\"ahler metrics on tangent bundles of
rank-one symmetric spaces of compact type}

\thanks{Research supported by the Ministry of Economy, Industry
and Competitiveness, Spain, under Project MTM2016-77093-P.}

\author[P.~M. Gadea]{P.~M. Gadea}
\address{Instituto de F\'\i sica Fundamental, CSIC,
Serrano 113 bis, 28006-Madrid, Spain.}
\email{p.m.gadea@csic.es }
\author[J.~C.~Gonz{\'a}lez-D{\'a}vila]{J.~C.~Gonz{\'a}lez-D{\'a}vila}
\address{Departamento de Matem\'aticas, Estad\'istica e Investigaci\'on
Ope\-ra\-tiva, University of La Laguna, 38200 La Laguna, Tenerife, Spain.}
\email{jcgonza@ull.es}
\author[I.~V. Mykytyuk]{I.~V. Mykytyuk}
\address{Institute of Mathematics, Cracow University of Technology,
Warszawska 24, 31155, Cracow, Poland. \newline
\indent Institute of Applied Problems of Mathematics and Mechanics,
Naukova Str. 3b, 79601, Lviv, Ukraine.}
\email{mykytyuk{\_\,}i@yahoo.com}


\maketitle

\begin{abstract}
We give an explicit description of
all complete $G$-invariant Ricci-flat K\"ahler metrics
on the tangent bundle $T(G/K)\cong G^\bbC/K^\bbC$ of
rank-one Riemannian symmetric spaces $G/K$ of compact type,
in terms of associated
vector-functions.
\end{abstract}

\section{Introduction}
\label{s.1}

Over the latest decades
there has been considerable interest in Ricci-flat K\"ahler
metrics whose underlying manifold is diffeomorphic to the
tangent bundle
$T(G/K)$ of a Riemannian symmetric space
$G/K$ of compact type. For instance, a remarkable class of Ricci-flat
K\"ahler manifolds of cohomo\-geneity one was discovered by
M.~Stenzel~\cite{St}. This has originated a great deal of papers.
To cite but a few: M.~Cveti\v c,
G.\,W.~Gibbons, H.~L\"u and C.\,N.~Pope~\cite{CGLP} studied
certain harmonic forms on these manifolds and found an
explicit formula for the Stenzel metrics in terms of
hypergeometric functions. Earlier, T.\,C.~Lee~\cite{Le} gave
an explicit formula of the Stenzel metrics for classical spaces
$G/K$ but in another vein, using the approach of G.~Patrizio
and P.~Wong~\cite{PW}. Remark also that in the case of the
standard sphere
${\mathbb S}^2$, the Stenzel metrics coincide with the
well-known Eguchi-Hanson metrics~\cite{EH}.
On the other hand, and as it is well known, Stenzel metrics continue
being a source
of results both
in physics and differential geometry. We cite here only to G.~Oliveira \cite{Oli}
and M.~Ionel and
T.\,A. Ivey \cite{II}.

In the present paper we give an {\it explicit\/} description of
all complete $G$-invariant Ricci-flat K\"ahler metrics
on the tangent bundle $T(G/K)$ of
rank-one Riemannian symmetric spaces $G/K$ of compact type
or, equivalently, on the complexification
$G^{\mathbb C}/K^{\mathbb C}$ of $G/K$. To this end,
reached in our main assertions (Theorem~\ref{th.4.1}
and its Corollary~\ref{co.4.3}), we use the method of our
article~\cite{GGM}, giving the result in terms of associated
vector-functions (see below in this introduction). In this article
it is also shown that
this set of metrics contains a new family of metrics which are not
$\partial\bar\partial$-exact if
$G/K \in \{{\mathbb C}{\mathbf P}^n, n\geqslant 1\}$,
and coincides with the set of $\partial\bar\partial$-exact
Stenzel metrics for any of the latter spaces $G/K$.

Remark here that until now, in the case of the space
${\mathbb C}{\mathbf P}^n$ $(n\geqslant 1)$,
all known Ricci-flat K\"ahler metrics were Calabi metrics,
so being hyper-K\"ahlerian and thus automatically Ricci-flat
(see O.~Biquard and P.~Gauduchon~\cite{BG1,BG2} and
E.~Calabi~\cite{Ca}). Since by A.~Dancer and
M.Y.~Wang~\cite[Theorem 1.1]{DW} any complete
$G$-invariant hyper-K\"ahlerian metric on
$G/K={\mathbb C}{\mathbf P}^n$ $(n\geqslant 2)$
coincides with the Calabi metric, our new metrics are not
hyper-K\"ahlerian.

Note also, that in~\cite{DW} the K\"ahler-Einstein
metrics on manifolds of
$G$-cohomogeneity one were classified but only under one
additional assumption: It is assumed that the isotropy
representation of the space
$G/H$ (see our notation below) splits into pairwise
inequivalent sub-representations. This condition is crucial
for the fact that the Einstein equation can be solved
(see~\cite[Theorem 2.18]{DW}). But this assumption fails,
for instance, for the symmetric space
${\mathbb C}{\mathbf P}^n$ $(n\geqslant2)$.

Let $G/K$ be a rank-one symmetric space of a compact
connected Lie group $G$. The tangent bundle
$T(G/K)$ has a canonical complex structure
$J^K_c$ coming from the $G$-equivariant diffeomorphism
$T(G/K) \to G^\bbC/K^\bbC$. The latter space is the above-mentioned
complexification of $G/K$.
In our paper~\cite{GGM} we
described, for such a $G/K$, all
$G$-invariant K\"ahler structures
$({\mathbf g}, J^K_c)$ which are moreover Ricci-flat on the
punctured  tangent bundle $T^+(G/K)$ of $T(G/K)$. This description
is based on the fact that $T^+(G/K)$ is the image of $G/H \times \bbR^+$
under certain $G$-equivariant diffeomorphism. Here
$H$ denotes the stabilizer of any element of
$T(G/K)$ in general position. Such $G$-invariant
K\"ahler and Ricci-flat K\"ahler structures are determined
completely by a unique vector-function
${\mathbf a}\colon \bbR^+ \to \fg_H$
satisfying certain conditions,
$\fg_H$ being the subalgebra of
$\Ad(H)$-fixed points of the Lie algebra of $G$.

As for the contents, we recall in Section \ref{s.2} some definitions and
results on the canonical complex structure on $T(G/K)$. In Section \ref{s.3}
we recall the general description given in \cite{GGM} of
invariant Ricci-flat K\"ahler metrics on tangent
bundles of Riemannian symmetric spaces of compact type,
especially in Theorems \ref{th.3.2} and \ref{th.3.5} below, given
here without proof. In Section \ref{s.4}, we state and prove
Theorem \ref{th.4.1} and its Corollary~\ref{co.4.3}
giving the invariant Ricci-flat K\"ahler metrics on the punctured
tangent bundles $T^+(G/K)$ of the rank-one Riemannian  symmetric spaces
of compact type and then the complete invariant Ricci-flat
K\"ahler metrics on $T(G/K)$.

\section{The canonical complex structure on $T(G/K)$}
\label{s.2}

Consider a homogeneous manifold $G/K$, where $G$ is a
compact connected Lie group and $K$ is
some closed subgroup of $G$. Let $\fg$ and
$\fk$ be the Lie algebras of $G$ and $K$ respectively.
There exists a positive-definite $\Ad(G)$-invariant form
$\langle\cdot ,\cdot \rangle$ on $\fg$.

Denote by
${\mathfrak m}$ the
$\langle\cdot ,\cdot \rangle$-orthogonal complement to
$\fk$ in $\fg$, that is,
$
\fg= {\mathfrak m} \oplus \fk
$
is the $\Ad(K)$-invariant vector space direct sum decomposition of
$\fg$. Consider the trivial vector bundle
$G\times {\mathfrak m}$ with the two Lie group actions
(which commute) on it:
the left
$G$-action, $l_h \colon(g,w)\mapsto(hg,w)$ and the right
$K$-action $r_k \colon(g,w)\mapsto(gk,\Ad_{k^{-1}}w)$. Let
$$
\pi \colon G\times {\mathfrak m}\to G\times_K {\mathfrak m},
\quad(g,w)\mapsto [(g,w)],
$$
be the natural projection for this right
$K$-action. This projection is
$G$-equivariant. It is well known that
$G\times_K {\mathfrak m}$ and
$T(G/K)$ are diffeomorphic. The corresponding
$G$-equivariant diffeomorphism
\begin{equation*}
\phi \colon G\times_K {\mathfrak m}\to T(G/K),
\quad
[(g,w)]\mapsto \frac{\rd}{\rd t}\Bigr|_0 g\exp(tw)K,
\end{equation*}
and the projection $\pi$ determine the $G$-equivariant submersion
$\Pi = \phi\circ\pi \colon G\times{\mathfrak m}\to T(G/K)$.

Let $G^{\mathbb C}$ and $K^{\mathbb C}$ be the complexifications
of the Lie groups $G$ and $K$.
In particular, $K$ is a maximal compact subgroup of the Lie group
$K^{\mathbb C}$ and the intersection of $K$ with each connected
component of
$K^{\mathbb C}$ is not empty~(cf.\ A.L.~Onishchik and E.V.~Vinberg
\cite[Ch.\ 5, p.\ 221]{OV} and note that
$G^{\mathbb C}$, $K^{\mathbb C}$, $G$ and $K$ are algebraic groups).
Let $\fg^\mathbb{C}=\fg\oplus\ai\fg$ and $\fk^\mathbb{C}=\fk\oplus\ai\fk$
be the complexifications of the compact Lie algebras
$\fg$ and $\fk$.

Since
$G$ and $K$ are maximal compact Lie subgroups of
$G^{\mathbb C}$ and $K^{\mathbb C}$,
respectively, by a result of G.D.~Mostow \cite[Theorem~4]{Mo1}, we have that
$K^{\mathbb C}=K\exp(\ai \fk)$,
$G^{\mathbb C}=G\exp(\ai {\mathfrak m})\exp(\ai \fk)$,
and the mappings
\begin{equation*}
\begin{split}
G\times {\mathfrak m}\times \fk &\to G^{\mathbb C}, \qquad
(g,w,\zeta) \mapsto g\exp(\ai w)\exp(\ai \zeta), \\
K\times \fk &\to K^{\mathbb C}, \qquad
(k,\zeta) \mapsto k\, \exp(\ai \zeta),
\end{split}
\end{equation*}
are diffeomorphisms. Then the map
\begin{equation*}
f^{\times}_K \colon G^{\mathbb C}/K^{\mathbb C}
\to G\times_K {\mathfrak m},
\qquad
g\exp(\ai w)\exp(\ai \zeta)K^{\mathbb C}\mapsto [(g,w)],
\end{equation*}
is a $G$-equivariant diffeomorphism~\cite[Lemma~4.1]{Mo2}.
It is clear that
\begin{equation*}
f_K \colon G^{\mathbb C}/K^{\mathbb C} \to T(G/K),
\qquad
g\exp( \ai w)\exp( \ai \zeta)K^{\mathbb C}\mapsto \Pi(g,w),
\end{equation*}
is also a $G$-equivariant diffeomorphism.
Since $G^{\mathbb C}/K^{\mathbb C}$ is a complex manifold,
the diffeomorphism
$f_K$ supplies the manifold $T(G/K)$ with the
$G$-invariant complex structure which we denote by $J_c^K$.

\section{Invariant Ricci-flat K\"ahler metrics on tangent
bundles of compact Riemannian symmetric spaces. General description}
\label{s.3}

We continue with the previous notations but in this section
and the next one it is assumed in addition that
$G/K$ is a rank-one Riemannian symmetric space of a
connected, compact semisimple Lie group $G$.

\subsection{Root theory of Riemannian symmetric spaces
of rank one}
\label{ss.3.1}

Here we will review a few facts about Riemannian symmetric
spaces of rank one~\cite[Ch.\ VII, \S2, \S11]{He} and
results of our paper~\cite{GGM} adapted to the case of
these (rank one) spaces.

We have then
\begin{equation*}
\fg=\fm\oplus\fk,
\quad\text{where }\quad
[{\mathfrak m},{\mathfrak m}]\subset\fk,
\quad [\fk,\fm]\subset \fm,
\quad [\fk,\fk]\subset \fk,
\quad \text{and} \quad \fk\bot \fm.
\end{equation*}
In other words, there exists an involutive automorphism
$\sigma \colon \fg\to\fg$ such that
\begin{equation*}
\fk=(1+\sigma)\fg
\quad \text{and} \quad
\fm=(1-\sigma)\fg.
\end{equation*}
Moreover, the scalar product $\langle\cdot ,\cdot \rangle$
is $\sigma$-invariant.

Let $\fa\subset\fm$ be some Cartan subspace of
the space ${\mathfrak m}$. There exists a $\sigma$-invariant Cartan
subalgebra $\ft$ of $\fg$ containing the commutative subspace $\fa$,
i.e.
\begin{equation*}
\ft=\fa\oplus\ft_0,
\;\; \text{where}\;\;
\fa=(1-\sigma)\ft, \;\;
\ft_0=(1+\sigma)\ft.
\end{equation*}
Then the complexification $\ft^\bbC$ is a Cartan subalgebra of the
reductive complex Lie algebra $\fg^\bbC$ and we have
the root space decomposition
\begin{equation*}
\fg^{\mathbb C}=\ft^\bbC\oplus\sum_{\alpha\in\Delta}
\tilde\fg_\alpha.
\end{equation*}
Here $\Delta$ is the root system of $\fg^\bbC$ with respect to
the Cartan subalgebra $\ft^\bbC$. For each $\alpha\in\Delta$ we have
\begin{equation*}
\tilde\fg_\alpha=\big\{\tilde\xi\in\fg^{\mathbb C}:
\ad_{\kern1pt \tilde t}\tilde\xi=\alpha(\tilde t)\tilde \xi,\;
\tilde t\in\ft^{\mathbb C}\big\}
\quad\text{and}\;\;
\dim_{\kern1pt \bbC} \tilde\fg_\alpha=1.
\end{equation*}
It is evident that the
centralizer $\tilde\fg_0$ of the space $\fa^\bbC$
in $\fg^\bbC$ is the subalgebra
\begin{equation}\label{eq.3.1}
\tilde\fg_0=\ft^\bbC\oplus\sum_{\alpha\in\Delta_0}
\tilde\fg_\alpha,
\end{equation}
where $\Delta_0=\{\alpha\in\Delta: \alpha|_{\fa^\bbC}=0\}$
is the root system of the reductive Lie algebra
$\tilde\fg_0$ with respect to its Cartan
subalgebra $\ft^\bbC$.

The set
$\Sigma=\{\lambda\in(\fa^\bbC)^*:
\lambda=\alpha|_{\fa^\bbC},\ \alpha\in \Delta\setminus\Delta_0\}$
is the set of restricted roots of the triple $(\fg,\fk, \fa)$, which is
independent of the choice of the $\sigma$-invariant Cartan subalgebra
$\ft$ containing the Cartan subspace $\fa$.

Since $G/K$ is a rank-one Riemannian symmetric space,
$\dim\fa=1$. Then the restricted root system is either
$\Sigma=\{\pm\varepsilon\}$ or
$\Sigma=\{\pm\varepsilon, \pm {\frac 12}\varepsilon\}$,
where $\varepsilon\in(\fa^{\mathbb C})^*$.
There exists a unique (basis) vector
$X\in\fa$ such that $\varepsilon(X)=\ai$, where,
since the algebra $\fg$ is compact,
$\alpha(\ft)\subset\ai\,\bbR$ for each $\alpha\in\Delta$.
It is clear that multiplying our scalar product
$\langle\cdot ,\cdot \rangle$ by a positive
constant we can suppose that $\langle X,X \rangle=1$.
For each $\lambda\in\Sigma$ define the linear function
$\lambda' \colon \fa\to\bbR$, by the relation $\ai\lambda'=\lambda$.
Note that then
\begin{equation*}
\langle X,X \rangle=1,
\qquad \ai\varepsilon'=\varepsilon,
\quad\text{and}\quad
\varepsilon'(X)=1.
\end{equation*}

Since the algebra
$\tilde\fg_0$ coincides with the centralizer of
the element $X\in\fa$ in
$\fg^\bbC$, there exists a basis
$\Pi$ of $\Delta$ (a system of simple roots) such that
$\Pi_0=\Pi\cap\Delta_0$ is a basis of
$\Delta_0$. Indeed, the element
$-\ai X\in\ai\ft$ belongs to the closure of some Weyl
chamber in $\ai\ft$ determining the basis $\Pi$. Then
$\Pi_0=\{\alpha\in\Pi: \alpha(-\ai X)=0\}$. The bases
$\Pi$ and $\Pi_0$ determine uniquely the subsets
$\Delta^+$ and $\Delta^+_0$ of positive roots of
$\Delta$ and
$\Delta_0$, respectively. It is evident that
\begin{equation*}
\Delta^+\setminus \Delta^+_0=\{\alpha\in \Delta: \alpha(-\ai X)>0\}.
\end{equation*}

The following decomposition
\begin{equation*}
\fg^{\mathbb C}=\tilde\fg_0\oplus
\sum_{\lambda\in\Sigma^+}
(\tilde\fg_\lambda\oplus\tilde\fg_{-\lambda}),
\quad \text{where}\quad
\tilde\fg_\lambda
= \sum_{\alpha\in\Delta\setminus\Delta_0,\, \alpha|_{\fa^\bbC}=\lambda}
\tilde\fg_\alpha
\end{equation*}
and $\Sigma^+$ denotes the subset
of positive restricted roots in $\Sigma$ determined by the
set of positive roots $\Delta^+$,
gives us a simultaneous diagonalization of
$\operatorname{ad}(\fa^{\mathbb C})$ on
$\fg^{\mathbb C}$.
Remark that in our case either
$\Sigma^+=\{\varepsilon\}$ or
$\Sigma^+=\{\varepsilon, {\frac 12}\varepsilon\}$.
Denote by
$m_\lambda$ the multiplicity of the restricted root
$\lambda\in\{\pm\varepsilon, \pm{\frac 12}\varepsilon\}$, that is,
$m_\lambda=\card\{\alpha\in\Delta: \alpha|_{\fa^\bbC}=\lambda\}$.

For each linear form $\lambda$ on $\fa^{\mathbb C}$ put
\begin{equation}\label{eq.3.2}
\begin{split}
{\mathfrak m}_\lambda&\eqdef
\big\{\eta\in{\mathfrak m}:\operatorname{ad}^2_w(\eta)
=\lambda^2(w)\eta,\ \forall w\in\fa\big\}, \\
\fk_\lambda&\eqdef
\big\{\zeta\in\fk: \operatorname{ad}^2_w(\zeta)
=\lambda^2(w)\zeta,\
\forall w\in\fa\big\}.
\end{split}
\end{equation}
Then ${\mathfrak m}_{\lambda}={\mathfrak m}_{-\lambda}$,
$\fk_{\lambda}=\fk_{-\lambda}$,
${\mathfrak m}_0=\fa$ and $\fk_0 = \fh$, where
\begin{equation}\label{eq.3.3}
\fh=\{u\in\fk: [u,\fa]=0\}=(\ker \ad_X)\cap\fk
\end{equation}
is the centralizer of $\fa$ in $\fk$.

In Table~3.1 we list all compact Riemannian symmetric
spaces of rank one with their corresponding multiplicities
$m_\varepsilon$, $m_{\varepsilon/{\scriptscriptstyle 2}}$
and type of the algebra $\fh$.

\bigskip
\begin{tabular}{llllll}
\multicolumn{6}{c}{Table 3.1. Irreducible rank-one Riemannian
symmetric spaces of compact type}\\
\hline\noalign{\smallskip}
& $G/K$
& {$\dim$}
& $m_\varepsilon$
& $m_{\varepsilon/{\scriptscriptstyle 2}}$ & $\fh$ \\
\noalign{\smallskip}\hline\noalign{\smallskip}
$\begin{matrix}
\mathbb{S}^n,(n\geqslant 2) \\ \noalign{\smallskip}
\hskip-14pt [{\mathbb R}{\mathbf P}^n]^*
\end{matrix}$ &
$\begin{matrix}
\mathrm{SO}(n{+}1)/\mathrm{SO}(n) \\ \noalign{\smallskip}
 [\mathrm{SO}(n{+}1)/\mathrm{O}(n)]^*\end{matrix}$ \smallskip&
$n$ & $n{-}1 $ & $0$ & $\mathfrak{so}(n{-}1)$ \\ \smallskip
${\mathbb C}{\mathbf P}^n$,\,{$(n \geqslant 2)$}
&{$\mathrm{SU}(n{+}1)/\mathrm{S}(\mathrm{U}(1){\times}
\mathrm{U}(n))$}& $2n$
& $1$ & $2n{-}2$ & $\bbR\oplus \mathfrak{su}(n{-}1)$ \\ \smallskip
${\mathbb H}{\mathbf P}^n$,\,{$(n\geqslant 1)$}
& {$\mathrm{Sp}(n{+}1)/\mathrm{Sp}(1){\times} \mathrm{Sp}(n)$} & $4n$
& $3$ & $4n{-}4$ & $\mathfrak{sp}(1)\oplus \mathfrak{sp}(n{-}1)$ \\
\smallskip
${\mathbb C\mathrm a}{\mathbf P}^2$
& {$\mathrm{F_4}/\mathrm{Spin(9)}$} & $16$ & $7$
& 8 & $\mathfrak{so}(7)$ \\
\hline
\end{tabular}

\bigskip\noindent
Here we assume that $\mathfrak{so}(1)=0$,
$\mathfrak{so}(2)\cong \bbR$, $\mathfrak{su}(0)
=\mathfrak{su}(1)=0$, $\mathfrak{sp}(0)=0$.
The symmetric spaces $G/K$ with non-connected $K$
are marked with $[\cdot]^*$ in Table~3.1.

It is clear that
$\fm^\bbC_\lambda\oplus\fk^\bbC_\lambda= \tilde\fg_\lambda
\oplus\tilde\fg_{-\lambda}$
for $\lambda\in\Sigma^+$ and
$\tilde\fg_0=\fm_0^\bbC\oplus\fk_0^\bbC=
\fa^\bbC\oplus\fh^\bbC$
(the Cartan subspace
$\fa^\bbC$ is a maximal commutative subspace of
$\fm^\bbC$). By~\cite[Ch.\ VII, Lemma 11.3]{He}, the
following decompositions are direct and orthogonal:
\begin{equation}\label{eq.3.4}
{\mathfrak m}=\fa\oplus
\fm_\varepsilon\oplus \fm_{\varepsilon/{\scriptscriptstyle 2}},
\qquad
\fk=\fh\oplus
\fk_\varepsilon\oplus \fk_{\varepsilon/{\scriptscriptstyle 2}},
\end{equation}
where to simplify the notation we suppose
that $\fm_{\varepsilon/{\scriptscriptstyle 2}}=0$
and $\fk_{\varepsilon/{\scriptscriptstyle 2}}=0$ if
$\tfrac12\varepsilon\not\in\Sigma$. We shall put
\[
{\mathfrak m}^+ \overset{\mathrm{def}}{=}
\fm_\varepsilon \oplus \fm_{\varepsilon/{\scriptscriptstyle 2}},
\qquad
\fk^+ \overset{\mathrm{def}}{=}
\fk_\varepsilon\oplus \fk_{\varepsilon/{\scriptscriptstyle 2}}.
\]
Since the restriction of the operator $\ad_X$ to the
subspace $\fm^+\oplus\fk^+$ is nondegenerate and
$\ad_X(\fm)\subset\fk$, $\ad_X(\fk)\subset\fm$
for any vector
$\xi_\lambda\in{\mathfrak m}_\lambda\subset\fm$,
$\lambda\in\Sigma^+$, by~(\ref{eq.3.2}) and~(\ref{eq.3.4})
there exists a unique vector
$\zeta_\lambda\in\fk_\lambda$ such that
\begin{equation}\label{eq.3.5}
[X,\xi_\lambda]= -\lambda'(X)\zeta_\lambda,
\quad [X,\zeta_\lambda]= \lambda'(X)\xi_\lambda,
\end{equation}
where, recall, $\varepsilon'(X)=1$. In particular,
$\dim{\mathfrak m}_\lambda =\dim\fk_\lambda=m_\lambda$
and there exists a unique endomorphism
$T \colon \fm^+ \oplus \fk^+ \to \fm^+ \oplus \fk^+$
such that
\begin{equation}\label{eq.3.6}
\ad_X|_{\fm_\lambda\oplus\fk_\lambda}
=\lambda'(X)T |_{\fm_\lambda\oplus\fk_\lambda},
\quad T({\fm_\lambda})={\fk_\lambda},
\quad T({\fk_\lambda})={\fm_\lambda},
\quad \forall \lambda\in\Sigma^+.
\end{equation}
This endomorphism is orthogonal because
$T^2= -\Id_{\fm^+ \oplus \fk^+}$ and the endomorphism
$\ad_X$ is skew-symmetric.
Note also here that by~(\ref{eq.3.1}) the subspace
\begin{equation*}
\ft_0=(1+\sigma)\ft
\quad\text{is a Cartan subalgebra of the centralizer}\ \fh
\ \text{and}\ \ft=\fa\oplus\ft_0.
\end{equation*}
Moreover, since $[\ft_0,\fm] \subset \fm$ and $[\ft_0,\fk] \subset \fk$,
$[\fa,\ft_0] = 0$, from definitions \eqref{eq.3.2}
and \eqref{eq.3.6} we obtain that
\begin{equation*}
\begin{split}
& [\ft_0,\fm_\lambda] \subset \fm_\lambda
\quad \text{and} \quad [\ft_0,\fk_\lambda] \subset \fk_\lambda
\quad \text{for each} \quad \lambda \in \Sigma^+, \\
& [\ad_x, T] = 0 \quad \text{on} \quad \fm^+ \oplus \fk^+
\quad \text{for each} \quad x \in \ft_0.
\end{split}
\end{equation*}

Fix the Weyl chamber $W^+$ in $\fa$ containing the element $X$:
\begin{equation*}
W^+=\bigl\{w\in\fa: \varepsilon(-\ai w)>0\bigr\}
=\bigl\{w\in\fa: \varepsilon'(w)>0\bigr\}=\bbR^+ X.
\end{equation*}
The subspace $\fm\subset\fg$ is
$\Ad(K)$-invariant. Each nonzero $\Ad(K)$-orbit in
$\fm$ intersects the Cartan subspace
$\fa$ and also the Weyl chamber
$W^+$, that is,
$\Ad(K)(W^+)=\fm\setminus\{0\}$. The set
$\fm^R=\fm\setminus\{0\}$ of all nonzero elements of
$\fm$ is the set of regular points in $\fm$.

Consider the centralizer $H$ of the Cartan subspace
$\fa$ in $\Ad(K)$, i.e.
\begin{equation}\label{eq.3.7}
H=\{k\in K: \Ad_k u=u \text{ for all } u\in\fa\}=\{k\in K: \Ad_k X=X\}.
\end{equation}
It is clear that the algebra $\fh$ (see~(\ref{eq.3.5})),
is the Lie algebra of $H$.

Our interest now centers on what will be shown to be an important
subalgebra of $\fg$. Let $\fg_H\subset\fg$ be the subalgebra
of fixed points of the group $\Ad(H)$, i.e.
\begin{equation}\label{eq.3.8}
\fg_H\eqdef \{u\in\fg: \Ad_h u=u
\text{ for all } h\in H\}.
\end{equation}
It is evident that $\fg_H\subset \fg_\fh$, where
\begin{equation}\label{eq.3.9}
\fg_\fh \eqdef\bigl\{u\in\fg: [u, \zeta]=0
\text{ for all } \zeta\in \fh\bigr\}
\end{equation}
is the centralizer of the algebra $\fh$ in $\fg$.
Note that in the general case one has
$\fg_H\ne \fg_\fh$ (see Example~4.6 in~\cite{GGM}).

To understand the structure of the algebra $\fg_H$
we consider more carefully the centralizer $\fg_\fh$.
Since $\fh$ is a compact Lie algebra,
$\fh=\fz(\fh)\oplus[\fh,\fh]$, where $\fz(\fh)$ is the center of
$\fh$ and $[\fh,\fh]$ is a maximal
semisimple ideal of $\fh$. It is clear that
\begin{equation*}
\fz(\fh)\subset \fg_\fh
\quad\text{and}\quad
\fg_\fh\cap [\fh,\fh]=0
\quad\text{because}\
\langle \fg_\fh,[\fh,\fh] \rangle
=\langle [\fg_\fh,\fh],\fh \rangle=0.
\end{equation*}
Therefore $\fg_\fh\cap \fh=\fz(\fh)$ and
$\fg_\fh\oplus[\fh,\fh]=\fg_\fh+\fh$ is a subalgebra of $\fg$.

By its definition, $\fz(\fh)$
is a subspace of the center of the algebra $\fg_\fh$.
Moreover, by~(\ref{eq.3.3}), $\fa\subset\fg_\fh$.
The space
$\fa\oplus \fz(\fh)\subset\fg_\fh$ is a Cartan subalgebra of
$\fg_\fh$ (a maximal commutative subalgebra of $\fg_\fh$)
because the centralizer of $\fa$ in $\fg$
equals $\fa\oplus\fh$,
$\fa\oplus \fz(\fh)$ is the center of the algebra $\fa\oplus\fh$
and $\fg_\fh\cap(\fa\oplus\fh)=\fa\oplus\fz(\fh)$ by definition of
$\fg_\fh$ (see also~\cite[Subsection 4.1]{GGM}).

Since $\fa\subset\fg_\fh$ and
$\ft_0\subset\fh$, then $\fa \oplus \ft_0\subset \fg_\fh + \fh$.
But $\fa\oplus \ft_0=\ft$ is a Cartan subalgebra of $\fg$.
This means that the complex reductive Lie algebras
$(\fg_\fh+\fh)^\bbC$, $\fg_\fh^\bbC$ and $\fh^\bbC$ are
$\ad(\ft^\bbC)$-invariant subalgebras of $\fg^\bbC$.
Taking into account that
$\ft\cap\fg_\fh=\fa\oplus\fz(\fh)$ and $\ft\cap\fh=\ft_0$,
we obtain the following direct sum decompositions:
\begin{equation}\label{eq.3.10}
\fg_\fh^\bbC=\fa^\bbC\oplus\fz(\fh)^\bbC\oplus
\sum_{\alpha\in\Delta_\fh} \tilde\fg_\alpha
\quad\text{and}\quad
\fh^\bbC=\ft_0^\bbC\oplus\sum_{\alpha\in\Delta_0}
\tilde\fg_\alpha,
\end{equation}
where $\Delta_\fh$ is some subset of the root system $\Delta$.
Since the spaces
$\fa\oplus\fz(\fh)\subset\ft$ and $\ft_0\subset\ft$
are Cartan subalgebras of the algebras $\fg_\fh$
and $\fh$ respectively, the decompositions above
are the root space decompositions of
$(\fg_\fh^\bbC,(\fa\oplus\fz(\fh))^\bbC)$ and
$(\fh^\bbC,\ft_0^\bbC)$, respectively. In particular, the subset
$\Delta_\fh\subset\Delta$ is the root system of
$(\fg_\fh^\bbC,(\fa\oplus\fz(\fh))^\bbC)$.

Since $\fh\subset\fk$, we see that
$\sigma(\fh)=\fh$ and the centralizer $\fg_\fh$ of
$\fh$ in $\fg$ is $\sigma$-invariant.
By~\cite[Proposition~4.3]{GGM},
\begin{equation*}
\Delta_\fh=\big\{\alpha\in\Delta:
\alpha(\ft_0)=0,\
\alpha+\beta\not\in\Delta\  {\rm for\ all }\ \beta\in\Delta_0
\big\}.
\end{equation*}
But by ~\cite[Lemma~4.1]{GGM} this subset
$\Delta_\fh$ of the set of roots
$\Delta$ admits the following alternative description:
\begin{equation}\label{eq.3.11}
\Delta_\fh=\big\{\alpha\in\Delta:
\alpha(\ft_0)=0, m_\lambda=1, \ \text{where}\
\lambda=\alpha|_{\fa^\bbC}\big\}.
\end{equation}
As follows from Table~3.1, two such restricted roots
$\{\varepsilon,-\varepsilon\}\subset\Sigma$ of multiplicity $1$
exist if and only if
$G/K\in\{{\mathbb C}{\mathbf P}^n \,(n\geqslant1),
{\mathbb R}{\mathbf P}^2\}$
(${\mathbb C}{\mathbf P}^1\cong \mathbb{S}^2$).
Hence for any of the latter rank-one symmetric spaces
$\fg_\fh=\fa\oplus\fz(\fh)$. Since for these latter spaces
$\fz(\fh)=0$ (see Table~3.1), we obtain that
\begin{equation}\label{eq.3.12}
\fg_\fh=\fa
\quad\text{if}\quad
G/K\not\in\{{\mathbb C}{\mathbf P}^n \,(n\geqslant1),
{\mathbb R}{\mathbf P}^2\}.
\end{equation}

Since $\sigma(\fh)=\fh$, the centralizer $\fg_\fh$ of
$\fh$ in $\fg$ is $\sigma$-invariant, i.e.
\begin{equation*}
\fg_\fh=\fm_\fh\oplus\fk_\fh,
\quad\text{where}\quad
\fm_\fh=\fg_\fh\cap\fm, \quad \fk_\fh=\fg_\fh\cap\fk
\end{equation*}
and as $\fa\subset\fm_\fh$ is a maximal commutative
subspace of $\fm$, the space $\fa$
is a Cartan subspace of $\fm_\fh$.
Then the set
$$\Sigma_\fh=\{\lambda\in(\fa^\bbC)^*:
\lambda=\alpha|_{\fa^\bbC},\ \alpha\in \Delta_\fh\}\subset\Sigma
$$
is the set of restricted roots of the triple
$(\fg_\fh,\fk_\fh, \fa)$ and since by~(\ref{eq.3.11}) the spaces
$\fm_\lambda$, $\lambda\in\Sigma_\fh$,
have dimension one,
we obtain the following direct orthogonal decompositions
\begin{equation*}
{\mathfrak m}_\fh=\fa\oplus
\sum_{\lambda\in\Sigma_\fh\cap \Sigma^+}{\mathfrak m}_\lambda,
\qquad
\fk_\fh=\fz(\fh)\oplus\sum_{\lambda\in\Sigma_\fh\cap \Sigma^+}
\fk_\lambda.
\end{equation*}

To describe the algebra $\fg_H\subset\fg_\fh$
we consider now in more detail the
subgroup $H\subset K$. By~\cite[Proposition~4.4]{GGM},
$H=(\exp(\fa)\cap K) H_0$, where $H_0=\exp\fh$ is the
identity component of the Lie group $H$
($H_0\subset K$ because $\fh\subset \fk$). Since the group
$H\subset K$ is compact and $K$ is a subgroup
of the group of fixed points of certain involutive automorphism
of $G$ acting by $\exp(v)\mapsto \exp(-v)$ on $\exp(\fa)$,
the discrete group $D_\fa\eqdef \exp(\fa)\cap K$ is finite and
\begin{equation}\label{eq.3.13}
D_\fa=\{\exp v: v\in \fa,\ \exp v=\exp(-v)\}\cap K.
\end{equation}
Since $[\fh,\fg_\fh]=0$, the group $\Ad(H_0)$ acts trivially
on $\fg_\fh$ and therefore
\begin{equation}\label{eq.3.14}
\fg_H=\{u\in\fg_\fh: \Ad_{\exp v} u=u \
\text{for all}\ v\in \fa \ \text{such that}\ \exp v\in D_\fa\}.
\end{equation}
Taking into account that $[\fa,\ft]=0$, we conclude that
the group $\Ad_{\exp\fa}$ acts trivially on the space
$\fa\oplus\fz(\fh)\subset\ft$ and consequently, by~(\ref{eq.3.12}),
\begin{equation}\label{eq.3.15}
\fg_H=\fg_\fh=\fa
\quad\text{if}\quad G/K\not\in\{{\mathbb C}{\mathbf P}^n \,(n\geqslant1),
{\mathbb R}{\mathbf P}^2\},
\end{equation}
and $\fg_H$ contains $\fa\oplus\fz(\fh)$ otherwise.
For the space $G/K={\mathbb C}{\mathbf P}^n$ $(n\geqslant1)$
we will calculate the algebra $\fg_H$ in the next section using
the matrix representation for $\fg \cong \mathfrak{su}(n+1)$.

The algebra $\fg_H$ is
$\sigma$-invariant because by definition~(\ref{eq.3.7}),
$\sigma\Ad(H)\sigma=\Ad(H)$. In particular,
\begin{equation*}
\fg_H=\fm_H\oplus\fk_H,
\quad\text{where}\quad
\fm_H=\fg_H\cap\fm, \quad \fk_H=\fg_H\cap\fk,
\end{equation*}
and $(\fg_H,\fk_H)$ is a symmetric pair. By maximality conditions
the space $\fa\subset\fg_H$ is a Cartan subspace of
$\fm_H\subset\fg_H$ and the space
$\fa\oplus\fz(\fh)$ is a Cartan subalgebra of
$\fg_H$.

For each $\lambda\in\Sigma^+$ and $g\in D_\fa\subset\exp\fa$ we have
that $\Ad_g(\fm_\lambda\oplus\fk_\lambda)=\fm_\lambda\oplus\fk_\lambda$
because $\Ad_{\exp v}=e^{\ad_v}$.
The set
\[
\Sigma_H=\{\lambda\in \Sigma_\fh:
\Ad_g|_{\fm_\lambda \oplus \fk_\lambda}
=\Id_{\fm_\lambda \oplus \fk_\lambda}
\ \text{for all}\ g\in D_\fa\}
\]
is the set of restricted roots of the triple
$(\fg_H,\fk_H, \fa)$. By~(\ref{eq.3.11}) each element
$\lambda\in \Sigma_H\subset\Sigma_\fh\subset\Sigma$
has multiplicity $1$ as an element of $\Sigma$, that is,
$\dim\fm_\lambda=\dim\fk_\lambda=1$.

The following decompositions are direct and orthogonal:
\begin{equation*}
{\mathfrak m}_H=\fa\oplus
\sum_{\lambda\in\Sigma_H\cap\Sigma^+}{\mathfrak m}_\lambda,
\qquad
\fk_H=\fz(\fh)\oplus\sum_{\lambda\in\Sigma_H\cap\Sigma^+}
\fk_\lambda.
\end{equation*}

\begin{remark}\label{re.3.1}
Put ${\mathfrak m}_H^+=\sum_{\lambda\in\Sigma_H\cap\Sigma^+}
{\mathfrak m}_\lambda$
and $\fk_H^+=\sum_{\lambda\in\Sigma_H\cap\Sigma^+}
\fk_\lambda$. Consider the orthogonal decompositions:
${\mathfrak m}^+={\mathfrak m}_H^+\oplus{\mathfrak m}_{\tw}^+$
and
$\fk^+=\fk_H^+\oplus\fk_{\tw}^+$,
where $\fm^+_* = \sum_{\lambda \in \Sigma^+\backslash \Sigma_H}
\fm_\lambda$
and $\fk^+_* = \sum_{\lambda \in \Sigma^+\backslash \Sigma_H}
\fk_\lambda$.
Since the decompositions
\begin{equation*}
\fg_H=\fa\oplus\fm_H^+\oplus\fk_H^+\oplus\fz(\fh),
\quad
\fg=\fa\oplus\fm_H^+\oplus\fk_H^+
\oplus\fm_{\tw}^+\oplus\fk_{\tw}^+\oplus\fh=
\fg_H\oplus(\fm_{\tw}^+\oplus\fk_{\tw}^+)\oplus[\fh,\fh]
\end{equation*}
are orthogonal and $[\fg_H,\fh]=0$, one has that $\fg_H\oplus[\fh,\fh]$
is a subalgebra of $\fg$.

Moreover, because of its definition,
$T({\fm_\lambda})={\fk_\lambda}$,
$T({\fk_\lambda})={\fm_\lambda}$
for all restricted roots $\lambda\in\Sigma^+$, we obtain that
\begin{equation*}
T(\fm_H^+)=\fk_H^+,
\quad
T(\fk_H^+)=\fm_H^+
\quad\text{and}\quad
T(\fm_{\tw}^+)=\fk_{\tw}^+,
\quad
T(\fk_{\tw}^+)=\fm_{\tw}^+.
\end{equation*}
\end{remark}

Fix in each subspace
$\fm_\lambda$, $\lambda\in\Sigma^+$, some basis
$\{\xi_\lambda^j, j=1, \dotsc ,m_\lambda\}$,
orthonormal with respect to the form
$\langle\cdot,\cdot \rangle$. In the case when
$\lambda\in\Sigma_\fh \cap\Sigma^+$,
$m_\lambda=1$ we have a unique vector
$\xi_\lambda^1$. As we remarked above, for each
$\lambda\in\Sigma^+$ there exists a unique basis
$\{\zeta_\lambda^j,\, j=1, \dotsc,m_\lambda\}$ of
$\fk_\lambda$ such that for each pair
$\{\xi_\lambda^j,\zeta_\lambda^j,
j=1, \dotsc ,m_\lambda\}$, condition~(\ref{eq.3.5}) holds.
The basis $\{\zeta_\lambda^j,
j=1, \dotsc,m_\lambda\}$, $\lambda\in\Sigma^+$, of $\fk_\lambda$,
is also orthonormal
due to the orthogonality of the operator $T$
(see~(\ref{eq.3.6})).
Fix also some orthonormal basis
$\{\zeta^k_0, k=1,\dotsc,\dim\fh\}$
of the centralizer $\fh$ of $\fa$ in $\fk$.
We will use the orthonormal basis
\begin{equation*}
X,\, \xi_\lambda^j,\,
\zeta_\lambda^j,\, j=1,\dotsc,m_\lambda, \,\lambda\in\Sigma^+;\;
\zeta^k_0,\; k=1,\dotsc,\dim\fh,
\end{equation*}
of the algebra $\fg$ in our calculations below.

\subsection{The canonical complex structure on $G/H\times W^+\cong
G/H\times \bbR^+$}
\label{ss.3.2}

By definition~(\ref{eq.3.7}) of the group $H$, the map
\begin{equation*}
K/H\times W^+\to \fm^R,\quad
(kH,w)\mapsto \Ad_k w,
\end{equation*}
is a well-defined diffeomorphism
because, recall, $W^+=\bbR^+ X$ and $\fm^R=\fm\setminus\{0\}$.
Thus the map
\begin{equation*}
f^+ \colon G/H\times W^+\to G\times_K\fm^R,
\quad
(gH,w)\mapsto[(g,w)],
\end{equation*}
is a well-defined $G$-equivariant diffeomorphism
of $G/H\times W^+$ onto the subset
$D^+=G\times_K\fm^R$,
which is an open dense subset of $G\times_K\fm$.

It is clear that the diagram
\begin{equation}\label{eq.3.16}
\begin{array}{rrl}
G\times W^+&\stackrel{\mathrm{id}}\longrightarrow&G\times \fm^R\\
\noalign{\smallskip}
{\downarrow\makebox[20pt]{\,$\scriptstyle{\pi_H\times \mathrm{id}}$}}
&&{\makebox[10pt]{$\scriptstyle{\pi}$}\downarrow} \\ \noalign{\smallskip}
G/H\times W^+&
\stackrel{f^+}\longrightarrow&G\times_K\fm^R
\end{array}
\hskip20pt
\begin{array}{rrl}
\makebox[35pt]{$(g,w)$\hfill}&\stackrel{\mathrm{id}}\longmapsto&(g,w)\\
\noalign{\smallskip}
{\downarrow\makebox[18pt]{\;$\scriptstyle{\pi_{H}\times \mathrm{id}}$}}
&&{\makebox[10pt]{$\scriptstyle{\pi}$}\downarrow} \\ \noalign{\smallskip}
(gH,w)&
\stackrel{f^+}\longmapsto&[(g,w)]
\end{array},
\end{equation}
where $\pi_H \colon G\to G/H$ is the canonical projection, is commutative.

Denote by $\xi^l$ the left $G$-invariant vector field on $G$
corresponding to $\xi\in\fg$. The submersion (projection)
$\pi \colon G\times\fm\to G\times_K\fm$ is (left)
$G$-equivariant. Therefore, the kernel
${\sK}\subset T(G\times {\mathfrak m})$ of the tangent map
$\pi_{*} \colon T(G\times\fm)\to T(G\times_K\fm)$
is generated by the global (left)
$G$-invariant vector fields $\zeta^L$, for
$\zeta\in \fk$, on $G\times {\mathfrak m}$,
\begin{equation}\label{eq.3.17}
\zeta^L{(g,w)}=(\zeta^{l}(g) ,[w,\zeta])\in T_g G\times T_w\fm,
\end{equation}
where the tangent space
$T_w\fm$ is canonically identified with the space $\fm$.

To describe the
$G$-invariant Ricci-flat K\"ahler metrics on $T(G/K)$ associated
to the cano\-nical complex structure $J^K_c$, we first attempt to
describe such metrics on the {\em punctured} tangent bundle
$T^+(G/K)\eqdef T(G/K)\setminus\{\mbox{zero section}\}$ of $G/K$.
It is clear that $T^+(G/K)=\phi(G\times_K \fm^R)$ and therefore
\begin{equation*}
T^+(G/K)=(\phi\circ f^+)(G/H\times W^+),
\end{equation*}
that is, $T^+(G/K)$ is $G$-equivariantly isomorphic to the direct product
$G/H\times W^+$, where the action of the group
$G$ on the first component is the natural one and that
on the second component is
the trivial one (see the commutative diagram~(\ref{eq.3.16})).
This $G$-equivariant diffeomorphism determines a
$G$-invariant complex structure on $G/H\times W^+$,
which  we denote also by $J^K_c$.

Note also here that the tangent space $T_o(G/H)$ at
$o=\{H\}\in G/H$ can be identified naturally with the
space
$
\fm\oplus\fk^+
= \fa\oplus\fm^+\oplus\fk^+,
$
because by definition $\fk=\fh\oplus\fk^+$
and $\fh$ is the Lie algebra of the group $H$.

Considering the coordinate $x$ on $W^+=\bbR^+ X$
associated with the basis vector $X$ of $\fa$,
we identify naturally $W^+\subset\fa$ with $\bbR^+$
replacing $w=xX$ by $x$:
\begin{equation*}
G/H\times W^+\to G/H\times \bbR^+,
\quad
(gH, xX)\mapsto(gH, x).
\end{equation*}

By the $G$-invariance it suffices
to describe the operators $J^K_c$ only at the
points $(o,x)\in G/H\times \bbR^+$, where $o=\{H\}$.
By~\cite[(4.47)]{GGM},
\begin{equation}\label{eq.3.18}
\begin{split}
J^K_c(o,x)(X,0)
&= \Big(0,\frac{\partial}{\partial x}\Big),\\
J^K_c(o,x)(\xi_\lambda^j,\, 0)
&=\Big(\frac{- \cosh\lambda'_x}{\sinh\lambda'_x}
\cdot\zeta_\lambda^j, \, 0\Big), \quad
j=1,\dotsc ,m_\lambda,\quad \lambda\in\Sigma^+,
\end{split}
\end{equation}
where $\lambda'_x=\lambda'(xX)\in\mathbb{R}$,
that is, $\lambda'_x=x$ if $\lambda=\varepsilon$ and
$\lambda'_x=\frac12 x$ if $\lambda=\frac12\varepsilon$.
Here
$T_o(G/H)$ is identified naturally with the space
$\fa\oplus \sum_{\lambda\in\Sigma^+}\fm_\lambda
\oplus \sum_{\lambda\in\Sigma^+}\fk_\lambda$, $\fa=\bbR X$, and,
in the first equation, we use naturally the usual
basis vector $\{\partial/\partial x\}$ of
$T_x\bbR^+$.

The second relation in~(\ref{eq.3.18}) can be represented
in a more general form (see~\cite[(4.27)]{GGM}):
\begin{equation*}
J^K_c(o,x)(\xi, 0)
=\Big(\frac{- \cos\ad_{xX}}{\sin\ad_{xX}}\xi,0\Big),
\quad\text{where }\xi\in\fm^+.
\end{equation*}

Let $F=F(J^K_c)$ be the subbundle of
$(1,0)$-vectors of the structure $J^K_c$ on the manifold
$G/H\times \bbR^+$. Since the map
$\pi_H\times \mathrm{id}\colon G\times \bbR^+\to G/H\times \bbR^+$
is a submersion, there exists a
unique maximal complex subbundle
$\sF$ of $T^\bbC(G\times \bbR^+)$ such that
$(\pi_H\times \mathrm{id})_*\sF=F$.
As shown in~\cite[(4.28),(4.29)]{GGM}, $\sF$
is generated by the kernel $\sH$ of the submersion
$\pi_H\times \mathrm{id}$,
\begin{equation}\label{eq.3.19}
\sH(g,x)=\{(\zeta^l(g),0),\,\zeta\in\fh\},
\quad g\in G, \;\; x\in \bbR^+,
\end{equation}
 and the left $G$-invariant global vector fields on $G\times \bbR^+$:
\begin{equation*}
\begin{split}
Z^{X}(g,x) &
= \Big(X^l(g),- \ri\, \frac{\partial}{\partial x}\Big),\\
\; Z^{\xi_\lambda^j}(g,x)
& =\left(\left(\frac{1}{\cosh\lambda'_x}\cdot\xi_\lambda^j- \ri \,
\frac{-1}{\sinh\lambda'_x}\cdot\zeta_\lambda^j\right)^{l}(g) ,0\right),
\end{split}
\end{equation*}
where $j=1,\dotsc,m_\lambda$, $\lambda\in\Sigma^+$.

To simplify calculations in the next subsection,
for the vector fields of the second family we will use
a more general expression
\begin{equation*}
Z^{\xi}(g,x)=\left(\left(R_x\xi- \ri \,
S_x\xi\right)^{l}(g) ,0\right),
\quad\xi\in\fm^+,
\end{equation*}
in terms of the two operator-functions
$R \colon \bbR^+\to\End(\fg)$
and $S \colon \bbR^+\to\End(\fg)$
on the set $\bbR^+$ such that
\begin{equation*}
\begin{split}
R_x\eta&=\frac{1}{\cos \ad_{xX}}\, \eta
\quad\text{if}\quad\eta\in\fm^+\oplus\fk^+,\qquad
R_x\eta=0 \quad\text{if}\quad\eta\in\fa\oplus\fh, \\
S_x\eta&=\frac{-1}{\sin \ad_{xX}}\, \eta
\quad\text{if}\quad\eta\in\fm^+\oplus\fk^+,\qquad
S_x\eta=0 \quad\text{if}\quad\eta\in\fa\oplus\fh,
\end{split}
\end{equation*}
where, recall, $xX\in W^+\subset\fa$.
Remark also that
$\frac{1}{\cos \ad_{xX}}\, \eta=\eta$ if $\eta \in \fa \oplus \fh$
but $R_x\eta=0$ in this case.
Since the operator $\ad_{xX}$ is skew-symmetric with
respect to the scalar product on $\fg$, each operator $R_x$
is symmetric and $S_x$ is skew-symmetric:
\begin{equation*}
\langle R_x\eta_1,\eta_2 \rangle=
\langle \eta_1,R_x\eta_2 \rangle,
\quad \langle S_x\eta_1,\eta_2 \rangle=
\langle \eta_1,-S_x\eta_2 \rangle,
\quad x\in \bbR^+, \ \eta_1,\eta_2\in\fg.
\end{equation*}
Moreover, since $xX\in W^+\subset\fa$, the restrictions
$R_x|_{\fm^+\oplus\fk^+}$ and
$S_x|_{\fm^+\oplus\fk^+}$ are nondegenerate and by Remark~\ref{re.3.1}
the following relations hold:
\begin{equation*}
R_x(\fm_s^+)=\fm_s^+, \quad
R_x(\fk_s^+)=\fk_s^+, \quad
S_x(\fm_s^+)=\fk_s^+, \quad
S_x(\fk_s^+)=\fm_s^+, \quad
s\in\{H,\tw\}.
\end{equation*}
It is clear also that
\begin{equation}\label{eq.3.20}
R_x|_{\fm_\lambda\oplus\fk_\lambda}=
\frac{1}{\cosh \lambda'_x}\, \Id_{\fm_\lambda\oplus\fk_\lambda},
\quad
S_x|_{\fm_\lambda\oplus\fk_\lambda}=
\frac{1}{\sinh \lambda'_x}\, T|_{\fm_\lambda\oplus\fk_\lambda}
\end{equation}
for all $\lambda\in\Sigma^+$, and $[R_x,T]=[S_x,T]=0$ on
$\fm^+\oplus\fk^+$ for all
$xX\in W^+$, where, recall, the operator
$T$ is defined by the expression~(\ref{eq.3.6}).

\subsection{Invariant Ricci-flat K\"ahler metrics on $G/H\times \bbR^+$}
\label{ss.3.3}

Let $\cK(G/H \times \bbR^+)=\{(\bfg,\omega, J^K_c)\}$ (resp.\
$\cR(G/H \times \bbR^+)=\{(\bfg,\omega, J^K_c)\}$) be the set of all
$G$-invariant K\"ahler (resp.\ Ricci-flat K\"ahler)
structures on $G/H \times \bbR^+$, identified also with the set
$\cK(T^+(G/K))$ (resp.\ $\cR(T^+(G/K))$) of all
$G$-invariant K\"ahler (resp.\ Ricci-flat K\"ahler)
structures on the open dense subset
$T^+(G/K)$ of $T(G/K)$, associated with $J^K_c$, via the
$G$-equivariant diffeomorphism $\phi \circ f^+\colon G/H \times \bbR^+
\to T^+(G/K)$ ($\bbR^+\cong W^+$).

Put
$$\{T_1,\dotsc,T_n\}=\{Z^X\}\cup
\{Z^{\xi_\lambda^j}, \,\lambda\in\Sigma^+,\,j=1, \dotsc, m_\lambda\}.
$$

The following theorem is Theorem~4.8 from~\cite{GGM} (adapted to
the rank one case) which describes the spaces $\cK(G/H \times \bbR^+)$ and
$\cR(G/H \times \bbR^+)$ in terms of invariant
forms on the space $G \times \bbR^+$:
\begin{theorem}\label{th.3.2}~{\em\cite{GGM}}
Let $\cK(G\times \bbR^+)=\{\tio\}$ be the set of all
$2$-forms $\tio$ on $G\times \bbR^+$ such that
\begin{mlist}
\item[$(1)$]
the form $\tio$ is closed;
\item[$(2)$]
the form $\tio$ is left $G$-invariant and right $H$-invariant;
\item[$(3)$]
the kernel of $\tio$ coincides with
the subbundle $\sH\subset T(G\times \bbR^+)$ in~{\em(\ref{eq.3.19});}
\item[$(4)$]
$\tio(T_j,T_k)=0$,\ $j,k=1,\dotsc,n;$
\item[$(5)$]
$\ri \, \tio(T,\overline{T})>0$ for each $T=\sum_{j=1}^n c_jT_j$,
where $(c_1,\ldots,c_n)\in\mathbb{C}^n\setminus\{0\}.$
\end{mlist}

Let $\cR(G\times \bbR^+)=\{\tio\}$ be the subset of the set
$\cK(G\times \bbR^+)=\{\tio\}$ consisting of all elements $\tio$
such that the following condition holds  {\em (}in addition{\em ):}
\begin{mlist}
\item[$(6)$]
$\det\bigl(\tio(T_j,\overline{T_k}) \bigr)=
\mathrm{const}$ on $G\times \bbR^+.$
\end{mlist}
Then {\em(i)}
For any $2$-form $\tio\in \cK(G\times \bbR^+)$
there exists a unique
$2$-form $\omega$ on $G/H\times \bbR^+\cong T^+(G/K)$ such that
$(\pi_H \times \mathrm{id})^\ast \omega =\tio$.
The map $\tio\mapsto \omega$ is a one-to-one map from
$\cK(G\times \bbR^+)$ onto $\cK(G/H\times \bbR^+)\cong \cK(T^+(G/K))$.

{\em(ii)} If the group $G$ is semisimple then the restriction of this map to
$\cR(G\times \bbR^+)$
is a one-to-one map from
$\cR(G\times \bbR^+)$ onto $\cR(G/H\times \bbR^+)\cong \cR(T^+(G/K))$.
\end{theorem}

\begin{remark}\label{re.3.3}
Note that condition $(5)$ of the previous theorem is equivalent
to the following condition: the
Hermitian matrix-function $\mathbf{w}(x)$ on
$\bbR^+$ with entries $w_{jk}(x)
=\ri \tio(T_j,\overline{T_k})(e,x)$, $j,k=1,\dotsc,n$,
is positive-definite.
\end{remark}

To prove that a K\"ahler structure on $T^+(G/K)$
admits a K\"ahler extension to the whole $T(G/K)$
we will use Corollary 4.10 from~\cite{GGM} (adapted to
the rank one case):
\begin{corollary}\label{co.3.4}~{\em\cite{GGM}}
Let $\omega\in\cK(G/H\times \bbR^+)$ and
$\tio=(\pi_H \times \mathrm{id})^\ast \omega $. Then
$\underline{\omega}=((\phi\circ f^+)^{-1})^*\omega\in \cK(T^+(G/K))$.
Suppose that there exists a smooth form (extension)
$\underline{\omega}_0$ on the whole tangent bundle $T(G/K)$
such that $\underline{\omega}_0=\underline{\omega}$ on
$T^+(G/K)$. Then the form $\underline{\omega}_0$ determines
a $G$-invariant K\"ahler structure
on $T(G/K)$ (associated
to the canonical complex structure $J^K_c$) if and only if
for some sequence $x_m\in \bbR^+$, $m\in\mathbb{N}$,
such that $\lim_{m\to\infty}x_m=0$,
the Hermitian matrix $\mathbf{w}(0)$ with entries
$w_{jk}(0)=\lim_{m\to\infty} w_{jk}(x_m)=
\lim_{m\to\infty} \ri \tio(T_j,\overline{T_k})(e,x_m) $, $j,k=1,\dotsc,n$,
is positive-definite.
\end{corollary}

\subsection{General description of the space $\cR(G\times \bbR^+)$}
\label{ss.3.4}

For any vector $a\in \fg$, denote by $\theta^a$ the left
$G$-invariant 1-form on the group $G$ such that
$\theta^a(\xi^{l})=\langle a,\xi \rangle$.
Since $r_g^*\theta^a=\theta^{\Ad_g a}$,
where $g\in G$, the form $\theta^a$ is
right $H$-invariant if and only if
$\Ad_h a=a$ for all $h\in H\subset G$.
Because
\begin{equation*}
\rd\theta^a(\xi^{l},\eta^{l})
=-\theta^a([\xi^{l},\eta^{l}])
=-\langle a,[\xi,\eta] \rangle,
\end{equation*}
the $G$-invariant form $\omega^a$ on $G$,
\begin{equation*}
\omega^a(\xi^{l},\eta^{l})\eqdef\langle a,[\xi,\eta] \rangle,
\quad
\xi,\eta\in\fg,
\end{equation*}
is a closed $2$-form on $G$.

Let $\mathrm{pr}_1 \colon G\times \bbR^+\to G$ and
$\mathrm{pr}_2 \colon G\times \bbR^+\to \bbR^+$
be the natural projections. Choosing some orthonormal basis
$\{e_1, \dotsc, e_N\}$ of the Lie algebra
$\fg$, where $e_1=X$, put
$\tith^{e_k}\eqdef \mathrm{pr} _1^*(\theta^{e_k})$ and
$\tio^{e_k}\eqdef \mathrm{pr} _1^*(\omega^{e_k})$.
For any vector-function $\mathbf{a} \colon \bbR^+\to\fg$,
$\mathbf{a}(x)=\sum_{k=1}^{N} a^k(x)e_k$, denote by
$\tith^{\mathbf{a}}$ (resp.\ $\tio^{\mathbf{a}}$)
the $G$-invariant 1-form $\sum_{k=1}^{N}a^k \cdot\tith^{e_k}$
(resp.\ 2-form $\sum_{k=1}^{N}a^k \cdot \tio^{e_k}$).

The following theorem~\cite[Theorem 5.1]{GGM}
(adapted to the rank one case) describes the spaces
$\cK(G\times \bbR^+)$ and $\cR(G\times \bbR^+)$
in terms of some $\bbR^+$-parameter family of exact $1$-forms on
the Lie group $G$:
\begin{theorem}\label{th.3.5}~{\em\cite{GGM}}
Let $\tio$ be a $2$-form belonging to
$\cK(G\times \bbR^+)$, where the compact Lie group
$G$ is semisimple. Then there exists a unique (up to a real
constant) smooth function
$f \colon \bbR^+\to \mathbb{R}$,
$x\mapsto f(x)$,
and a unique smooth vector-function
$\mathbf{a} \colon \bbR^+\to \fg_H$ given by
\begin{equation}\label{eq.3.21}
\begin{split}
& \mathbf{a}(x) =a^\fa(x)
+z_\fh+a^\fk(x)+a^\fm(x), \ \text{where}
\quad a^\fa(x)=f'(x) X,\ z_\fh\in\fz(\fh), \;\;\\
& a^\fk(x)=\sum_{\lambda\in\Sigma_H\cap\Sigma^+}
\tfrac{c^\fk_\lambda}{\cosh \lambda'(xX)}\zeta^{1}_\lambda
\in \fk_H^+, \;\;
a^\fm(x)=\sum_{\lambda\in\Sigma_{H}\cap\Sigma^+}
\tfrac{c^\fm_\lambda}{\sinh \lambda'(xX)}\xi^{1}_\lambda
\in \fm_H^+,
\end{split}
\end{equation}
$c^\fm_\lambda, c^\fk_\lambda\in\bbR$, such that $\tio$ is the
exact form expressed
in terms of $\mathbf a$ as
\begin{equation}\label{eq.3.22}
\tio=\rd \tith^{\mathbf{a}}
=\rd x \land \tith^{\mathbf{a'}} -\tio^{\mathbf{a}},
\quad \text{where}\quad
\mathbf{a'}=\frac{\partial \mathbf{a}}{\partial x}.
\end{equation}
Moreover, for all points $x\in \bbR^+$, the following conditions
$(1)\!\!-\!\!(3)$ hold:
\begin{mlist}
\item
[$(1)$] the components $a^\fk(x)+z_\fh$ and $a^\fm(x)$ of
the vector-function ${\mathbf a}(x)$ in~\emph{(\ref{eq.3.21})}
satisfy the commutation relations
\begin{equation}\label{eq.3.23}
\begin{split}
\bigl(R_x \cdot \ad_{a^\fk(x)} \cdot R_x
+S_x \cdot \ad_{a^\fk(x)} \cdot S_x
+(R^2_x + S^2_x)\ad_{z_\fh}\big)(\fm^+) & =0, \\
\bigl(R_x \cdot \ad_{a^\fm(x)} \cdot S_x-
S_x \cdot \ad_{a^\fm(x)} \cdot R_x\bigr)(\fm^+) & =0;
\end{split}
\end{equation}
moreover, if $G/K$
is an irreducible Riemannian symmetric space and $a^\fk(x)\equiv 0$,
then $z_\fh=0$;
\item
[$(2)$]
the Hermitian $p\!\times\! p$-matrix-function
${\mathbf w}_H(x)=\bigl(w_{k|j}(x)\bigr)$, $p=\dim\fm_H=1+\linebreak
\card(\Sigma_H\cap\Sigma^+)$,
with indices $k,j\in\{1\}
\cup\{{}_\lambda^{1},\, \lambda\in \Sigma_H\cap\Sigma^+\}$
and entries
\begin{equation*}
\begin{split}
w_{1|1}(x)&=2f''(x), \\
w_{1|{}_\lambda^{1}}(x)
&=2\lambda'(X)\Bigl(\ri \dfrac{c^\fk_\lambda}{\cosh^2 \lambda'_x}
-\dfrac{c^\fm_\lambda}{\sinh^2 \lambda'_x} \Bigr), \
\lambda\in \Sigma_H\cap\Sigma^+, \\
w_{{}_\lambda^{1}|{}_\mu^{1}}(x), & \quad
\lambda,\mu\in \Sigma_H\cap\Sigma^+, \quad
\text{determined by}~(\ref{eq.3.24}),
\end{split}
\end{equation*}
 is positive-definite;
\item
[$(3)$]
if $\fm_{\tw}^+\ne 0$ then
the Hermitian $s\times s$-matrix
${\mathbf w}_{\tw}(x)=(w_{{}_\lambda^j|{}_\mu^k})(x)$,
where $s=\dim\fm_{\tw}^+= \Sigma_{\lambda \in \Sigma^+ \setminus
\Sigma_H} m_{\lambda}$,
with indices ${}_\lambda^j,\, {}_\mu^k\in
\{{}_\lambda^j,\, \lambda\in\Sigma^+\setminus \Sigma_H,
j=1,\ldots,m_\lambda\}$ and entries
\begin{align}\label{eq.3.24}
w_{{}_\lambda^j|{}_\mu^k}(x)
& = -\frac{2\ri}{\sinh \lambda'_x \sinh \mu'_x}\langle
(\ad_{a^\fk(x)+z_\fh})\zeta^j_\lambda, \, \zeta^k_\mu\rangle \\
& \quad - \frac{2}{\cosh \lambda'_x \sinh \mu'_x}\langle
(\ad_{a^\fa(x)+a^\fm(x)})\xi^j_\lambda, \, \zeta^k_\mu\rangle \notag
\end{align}
is positive-definite.
\end{mlist}
If in addition
\begin{mlist}
\item
[$(4)$] either $\det {\mathbf w}_H(x)\cdot
\det {\mathbf w}_{\tw}(x) = \mathrm{const}$
when $\fm^+_{\tw}\ne0$ or $\det {\mathbf w}_H(x)\equiv
\mathrm{const}$ otherwise,
\end{mlist}
then $\widetilde\omega \in \cR(G \times \bbR^+)$.

Conversely, any $2$-form as in~\emph{(\ref{eq.3.22})}
determined by a vector-function
$\mathbf{a} \colon \bbR^+\to \fg_H$ as in~\emph{(\ref{eq.3.21})}
for which conditions $(1)\!\!-\!\!(3)$ hold, belongs to
$\cK(G\times \bbR^+)$ and if in addition $(4)$ holds,
it belongs to $\cR(G\times \bbR^+)$.
\end{theorem}

Also Theorem~\ref{th.3.2} immediately implies
\begin{corollary}\label{co.3.6}~{\em\cite{GGM}}
Let $G/K$ be a rank-one Riemannian symmetric space of compact type.
Each $G$-invariant K\"ahler metric ${\bf g}$, associated with
the canonical complex structure $J^K_c$ on
$G/H \times \bbR^+ \cong T^+(G/K)$ $($$T^+(G/K)$ is an open
dense subset of $T(G/K)$$)$,
is uniquely determined by the K\"ahler form
$\omega(\cdot, \cdot) = {\bf g}(-J^K_c \cdot, \cdot)$
on $G/H \times \bbR^+$ given by
\[
(\pi_H \times \mathrm{id})^\ast \omega = \rd\widetilde\theta^{\mathbf{a}},
\]
where $\mathbf{a}$ is the unique smooth vector-function
$\mathbf{a} \colon \bbR^+ \to \fg_H$ in \eqref{eq.3.21} satisfying
conditions
$(1)\!\!-\!\!(3)$ of \emph{Theorem~\ref{th.3.5}}.
If, in addition, condition $(4)$ of \emph{Theorem~\ref{th.3.5}} holds,
this metric ${\mathbf g}$ is Ricci-flat.
\end{corollary}

\begin{corollary}\label{co.3.7}~{\em\cite{GGM}}
Let $\omega$ be a $G$-invariant symplectic form on
$G/H\times \bbR^+$ such that
$(\pi_H \times \mathrm{id})^\ast \omega= \rd\tilde\theta^{\mathbf{a}}$,
where $\mathbf{a} \colon \bbR^+\to \fa$, $\mathbf{a}(x)= f'(x) X$,
for some function $f\in C^\infty(\bbR^+,\mathbb{R})$.
Then the pair
$(\omega,J^K_c)$ is a K\"ahler structure
on $G/H\times \bbR^+$
(equivalently $(\pi_H \times \mathrm{id})^\ast \omega \in
\cK(G\times \bbR^+)$$)$
if and only if $f'(x)>0$ and $f''(x)>0$ for all $x\in \bbR^+$.
In this case, the $G$-invariant function $Q \colon
G/H\times \bbR^+\to \bbR$, $Q(gH,x)=2f(x)$,
is a potential function of the K\"ahler structure $(\omega,J^K_c)$
on $G/H\times \bbR^+$.

The K\"ahler structure $(\omega,J^K_c)$
with $G$-invariant potential function
$Q$ is Ricci-flat K\"ahler
(equivalently $(\pi_H \times \mathrm{id})^\ast \omega \in
\cR(G\times \bbR^+)$$)$
if and only if
\begin{equation*}
f''\cdot\prod_{\lambda\in\Sigma^+}
\left( \frac{2\lambda'(\mathbf{a})}{\sinh 2\lambda'(\mathbf{a})}
\right)^{m_\lambda}=
f''
\cdot
\left( \frac{2f'}{\sinh(2f')}
\right)^{m_\varepsilon}
\cdot
\left( \frac{f'}{\sinh(f')}
\right)^{m_{\varepsilon/{\scriptscriptstyle 2}}}
\equiv\mathrm{const}.
\end{equation*}
\end{corollary}

\section{Complete invariant Ricci-flat K\"ahler metrics on
tangent bundles of rank-one Riemannian  symmetric spaces of compact type}
\label{s.4}

Let $\fg$ be a compact Lie algebra and let $\sigma$,
$\fk$, $\fm$, $\fa$, $X\in \fa$, $\Sigma$, etc.\ be as in
Section~\ref{s.3}. We continue with the previous
notations but in this section it is assumed in addition that
the subgroup $K$ is connected.

In this Section using Theorem~\ref{th.3.5}
we describe all invariant Ricci-flat K\"ahler structures on the tangent bundles
of the spaces under study, in terms of explicit expressions of the corresponding
vector-valued functions {\bf a}.

To this end we give with more detail
the facts concerning the case
$G/K={\mathbb C}{\mathbf P}^n$ $(n\geqslant1)$.
These spaces are Hermitian symmetric spaces and therefore we
will review a few facts about them~\cite[Ch.\ VIII,
\S\S4--7]{He}. The compact Lie subalgebra
$\fk$ of the semisimple Lie algebra
$\fg=\mathfrak{su}(n+1)$ is the direct sum
$\fk = \fz\oplus [\fk,\fk]$ of the one-dimensional center
$\fz$ and the semisimple ideal
$[\fk,\fk]\cong \mathfrak{su}(n)$. The subalgebra
$\fk$ coincides with the centralizer of
$\fz$ in $\fg$. Here
$\mathfrak{su}(n+1)$ denotes the space of traceless skew-Hermitian
$(n+1)\times(n+1)$ complex matrices and
$\fk=\{(b_{jk})\in \mathfrak{su}(n+1): b_{1j}=b_{j1}=0,\, j= 2,\dotsc, n+1\}$.
Fix on $\fg=\mathfrak{su}(n+1)$ the invariant trace-form given by
$\langle B_1, B_2 \rangle= -2\tr B_1 B_2$,
$B_1,B_2\in \mathfrak{su}(n+1)$. There exists a unique (up to a sign)
element $Z_0\in\fz(\fk)$ such that the endomorphism
$I=\operatorname{ad}_{Z_0}|_{\mathfrak m} \colon {\mathfrak
m}\to{\mathfrak m}$
satisfies $I^2=-\operatorname{Id}_{\mathfrak m}$. Choose
$Z_0$ as
\begin{equation}\label{eq.4.1}
Z_0={\rm diag}(i b_0, i(b_0-1),\ldots, i(b_0-1)),\quad
b_0=n/(n+1).
\end{equation}
By the invariance of the form
$\langle\cdot,\cdot\rangle$ on $\fg$, the form
$\langle\cdot,\cdot\rangle|_{\mathfrak m}$ is
$I$-invariant. Moreover, by the Jacobi identity,
\begin{equation}\label{eq.4.2}
[I\xi,I\eta]=[\xi,\eta], \qquad I[\zeta,\eta]=[\zeta,I\eta]
\qquad
\text{for all } \;\xi,\eta\in{\mathfrak m},\ \zeta\in\fk.
\end{equation}
Denote by $E_{jk}$ the elementary
$(n+1)\times(n+1)$ matrix whose entries are $0$ except for $1$ at the
entry in the $j$th row and $k$th column.
Choose as basis vector $X\in\fa$ the matrix
$X=\tfrac{1}{2} E_{12}-\tfrac{1}{2} E_{21}
\in{\mathfrak m}\subset \mathfrak{su}(n+1)$.
We will show below (using direct matrix calculations)
that this choice is consistent with the notation of the previous
sections, i.e.\ in this case $\langle X,X \rangle=1$
and the restricted root system
$\Sigma$ of $(\fg,\fk,\fa)$ coincides with the set
$\{\pm\varepsilon\}$ if $n=1$ and
$\{\pm\varepsilon,\pm\frac12 \varepsilon\}$ if $n\geqslant2$.
The center $\fz(\fh)$ of the centralizer
$\fh=\fg_{X}\cap\fk$
of $X\in\fa$ in $\fk$ is trivial for $n = 1$ and
one-dimensional for $n\geqslant 2$ (see Table~3.1).
It is easy to verify that $\fz(\fh)={\mathbb R}Z_1$, where
\begin{equation}\label{eq.4.3}
Z_1={\rm diag}\bigl(i b_1,i b_1, i(b_1-1),\ldots, i(b_1-1)\bigr),\quad
b_1=(n-1)/(n+1),\ n\geqslant 2.
\end{equation}
Note that $Z_{1} = 0$ for $n = 1$.

\begin{theorem}\label{th.4.1}
Let $G/K$ be a rank-one Riemannian symmetric
space of compact type with $K$ connected. A $2$-form
$\omega$ on the punctured tangent bundle
$T^{+}(G/K)$ of $G/K$ determines a $G$-invariant
K\"ahler structure, associated to the canonical complex structure
$J^{K}_{c}$, and the corresponding metric
${\mathbf g} = \omega(J^{K}_{c}\cdot,\cdot)$
is Ricci-flat, if and only if the
$2$-form
$\widetilde{\omega} = \bigl((\phi\circ f^+)\circ(\pi_{H}\times
{\rm id})\big)^{*}\omega$ on
$G\times {\mathbb R}^{+}$ may be expressed as
$\widetilde{\omega} = {\rm d}\widetilde{\theta}^{\mathbf a}$,
where
\begin{enumerate}
\item [$(1)$]
for $G/K\in\{ \mathbb{S}^n(n \geqslant 3),
\, {\mathbb H}{\mathbf P}^n(n\geqslant 1),
\,{\mathbb C\mathrm a}{\mathbf P}^2 \}$
the vector-function $\mathbf{a}(x)=f'(x)X$, where
\begin{equation}\label{eq.4.4}
\bigl(f'(x)\bigr)^{m_\varepsilon+m_{\varepsilon/{\scriptscriptstyle 2}}+1}
=C\cdot \int_0^x(\sinh 2t)^{m_\varepsilon}
(\sinh t)^{m_{\varepsilon/{2}}} {\rm d} t+C_1,
\end{equation}
$C,C_1\in{\mathbb R}, \; C > 0, \; C_1 \geqslant 0;$
\item [$(2)$]
for $G/K\in\{{\mathbb C}{\mathbf P}^n(n\geqslant 1)\}$
the vector-function is
\[
 \mathbf{a}(x)=f'(x)X+\frac{c_Z}{\cosh x}[IX,X] -\frac{1}{2}c_{Z}Z_{1},
 \]
where $c_Z$ is an arbitrary real number and
\begin{equation}\label{eq.4.5}
f'(x)=\sqrt{(C^n\sinh^{2n}x+C_1)^{1/n}
+c_Z^2\sinh^2 x\cosh^{-2} x},
\end{equation}
$C,C_1\in {\mathbb R}$, $C>0$,
$C_1\geqslant 0$.
\end{enumerate}
The corresponding $G$-invariant Ricci-flat K\"ahler metric
${\mathbf g} = {\mathbf g}(C,C_{1},c_{Z})$ on
$T^{+}(G/K)$ is uniquely extendable to a smooth complete
metric on the whole tangent bundle
$T(G/K)$ if and only if $C_{1} = 0$ (that is,
$\lim_{x\to 0}f'(x) = 0)$.
\end{theorem}

\noindent {\it Proof.}
By Theorem~\ref{th.3.5} we have to describe all vector-functions
$\mathbf{a} \colon {\mathbb R}^+\to \fg_{H}$
satisfying conditions
$(1)\!\!-\!\!(4)$ of that theorem. Then the 2-form
$\widetilde{\omega}= {\rm d}\widetilde{\theta}^{\,\mathbf{a}}$
belongs to the space ${\mathcal R}(G\times \bbR^+)$.
We consider the following two cases:

\medskip
{\bf(1)}
$G/K\in\bigl\{ \mathbb{S}^n(n\geqslant 3), \, {\mathbb
H}{\mathbf P}^n(n\geqslant 1), \,{\mathbb C\mathrm
a}{\mathbf P}^2 \bigr\}$.
In this case by~(\ref{eq.3.15})
$\fg_H=\fg_\fh=\fa$.
One gets that ${\mathfrak m}_{H} = \fa$ and
$\fk_{H} = 0$. Then
$\mathbf{a}(x)=f'(x)X$, $x\in {\mathbb R}^{+}$.
Let us describe the Hermitian matrix-functions
${\mathbf w}_{H}(x)$ and ${\mathbf w}_{*}(x)$
from Theorem~\ref{th.3.5}.
As it is easily seen, the first matrix ${\mathbf w}_{H}(x)$
contains a unique element
$w_{1|1}(x)=2f''(x)$ and the second one,
${\mathbf w}_{*}(x)$, is diagonal with elements
\begin{equation}\label{eq.4.6}
w_{{}_\varepsilon^j|{}_\varepsilon^j}(x)=
\tfrac{2 f'(x)}{\cosh x\cdot \sinh x},
\quad
w_{{}_{\varepsilon/{\scriptscriptstyle 2}}^k|
{}_{\varepsilon/{\scriptscriptstyle 2}}^k}(x)=
\tfrac{f'(x)}{\cosh \frac12 x\cdot \sinh\frac12 x},
\end{equation}
where $j=1,\dotsc,m_\varepsilon$
and $k=1,\dotsc,m_{\varepsilon/{\scriptscriptstyle 2}}$.
These matrices are positive definite if and only if
$f'(x)>0$ and $f''(x)>0$ for all $x\in\bbR^+$.
Hence the vector-function
$\mathbf{a} \colon {\mathbb R}^+\to\fa$
satisfies conditions
$(1)\!\!-\!\!(4)$ of Theorem~\ref{th.3.5}  (see also
Corollary~\ref{co.3.7}) if and only if
\[
\textstyle
f'(x)>0, \quad
f''(x)>0, \quad
f''(x)\cdot\Bigl(\frac{f'(x)}{\cosh x \sinh x} \Bigr)^{m_\varepsilon}
\Bigl(\frac{f'(x)}{\cosh \frac12 x
\sinh \frac12 x} \Bigr)^{m_{\varepsilon/{2}}}
\equiv\mathrm{const}, \
x\in{\mathbb R}^+.
\]
It is clear that the unique possible solution of these equations
is of form~(\ref{eq.4.4}).

\medskip
{\bf(2)} $G/K={\mathbb C}{\mathbf P}^n$ $(n\geqslant 2)$.
Theorem \ref{th.3.5} was shown for
$G/K={\mathbb C}{\mathbf P}^1\cong\mathbb{S}^2$
in our paper~\cite[Theorem 6.1]{GGM}.
Therefore in this proof we will suppose that $n\geqslant2$.
Since we have chosen the matrix
$X=\tfrac{1}{2} E_{12}-\tfrac{1}{2} E_{21}\in{\mathfrak m}
\subset \mathfrak{su}(n+1)$ as the basis vector $X\in\fa$,
it follows that
\begin{equation}\label{eq.4.7}
\begin{split}
Y & \eqdef IX =[Z_{0},X] = \tfrac{i}{2} E_{12}+\tfrac{i}{2} E_{21}\in\fm
\\
Z & \eqdef[IX,X] = -\tfrac{i}{2} E_{11}+\tfrac{i}{2} E_{22}\in\fk.
\end{split}
\end{equation}
It is easy to verify that
the set $\{X,Y,Z\}$ is an orthonormal system of vectors
in $\fg$ and
\begin{equation}\label{eq.4.8}
[X,Y]=-Z, \qquad [X,Z]=Y, \qquad [Z,Y]=X,
\end{equation}
i.e.\ the vectors $\{X,Y,Z\}$ form
a canonical basis of the Lie algebra isomorphic to
$\mathfrak{su}(2)$. By~(\ref{eq.4.8}),
$\ad_X^2(IX)$ $=-IX$ and as it is easy to verify,
$\ad_X^2(\xi)=-\frac14\xi$ for any vector
$\xi$ from the set of vectors
\begin{equation}\label{eq.4.9}
\xi^{2j-1}_{\varepsilon/{\scriptscriptstyle 2}}=
\tfrac{1}{2} E_{1(2+j)}-\tfrac{1}{2} E_{(2+j)1},
\quad
\xi^{2j}_{\varepsilon/{\scriptscriptstyle 2}}=
\tfrac{i}{2} E_{1(2+j)}+\tfrac{i}{2} E_{(2+j)1},
\quad
j= 1, \dotsc, n-1.
\end{equation}
Defining the restricted root
$\varepsilon\in(\fa^{\mathbb C})^*$
by the relation $\varepsilon'(X)=1$
($\varepsilon(X)=i$), we obtain that
${\mathfrak m}_\varepsilon={\mathbb R}(IX)$
and that the set~(\ref{eq.4.9}) is an orthonormal basis of the space
${\mathfrak m}_{\varepsilon/{\scriptscriptstyle 2}}$
of dimension $2n-2$ (the orthogonal complement to
$\fa\oplus{\mathfrak m}_\varepsilon$ in
${\mathfrak m}$). Moreover,
$I\xi^{2j-1}_{\varepsilon/{\scriptscriptstyle 2}}=
\xi^{2j}_{\varepsilon/{\scriptscriptstyle 2}}$
for each $j= 1, \dotsc, n-1$. Therefore
$\Sigma^+=\{\varepsilon,\frac12 \varepsilon\}$
($n\geqslant 2$).

Let us calculate the subalgebras
$\fg_\fh$ and $\fg_H$ of $\fg$ determined by
relations~(\ref{eq.3.8}) and~(\ref{eq.3.9}).
By~(\ref{eq.3.10}) the space $\fa\oplus\fz(\fh)$,
where $\fa=\bbR X$ and $\fz(\fh)=\bbR Z_1$,
is a Cartan subalgebra of the algebra $\fg_\fh$.
Since $\fz(\fh)$ belongs to the center of $\fg_\fh$
we see that $\rank[\fg_\fh, \fg_\fh]\leqslant \dim\fa=1$,
that is, $\fg_\fh \cong \mathfrak{su}(2)\oplus\fz(\fh)$
if the algebra $\fg_\fh$ is not commutative.

Since by definition $[X,\fh]=0$ and
$\fh\subset\fk$, by~(\ref{eq.4.2}) one gets that
$[IX,\fh]=0$, that is,
$IX\in \fg_\fh$.
By~(\ref{eq.4.8}), the subalgebra in $\fg$
generated by the vectors $X$ and $IX=Y$ is not commutative.
Thus $\fg_\fh$ is not commutative and, consequently,
$\fg_\fh \cong \mathfrak{su}(2)\oplus\fz(\fh)$.
The vectors $\{X,Y,Z,Z_1\}$ form an orthonormal basis of $\fg_\fh$.
Therefore $\Sigma_\fh = \{\pm\varepsilon\}$.

Let us find now
the algebra $\fg_{H}$.
The finite group
$D_\fa$ defined by relation (\ref{eq.3.13}) is given by
\[
\begin{split}
D_\fa
&= \{\exp tX : \exp tX= \exp(-tX)\}\cap K \\
&= \{\exp 4\pi X, \exp 2\pi X\}=\{\diag(1,\dotsc,1),\
\diag(-1,-1,1,\dotsc,1)\}.
\end{split}
\]
It is clear that the group $\Ad(D_\fa)$ acts trivially
on the space generated by the vectors $X,Y,Z$ and $Z_1$.
Therefore by~(\ref{eq.3.14})
we have that $\fg_H=\fg_\fh$
and, consequently, $\Sigma_H=\Sigma_\fh=\{\pm\varepsilon\}$.
Note that $D_\fa\subset H_0\cong
\mathrm{U}(1){\times}\mathrm{SU}(n-1)$,
$\mathrm{U}(1)\cong\{\exp tZ_1, t\in\bbR\}$,
i.e. the subgroup $H$ is connected.

Using properties~(\ref{eq.4.2}) of the automorphism
$I$, we can describe the actions of the operators
$\ad_Y$ and $\ad_Z$ on
${\mathfrak m}\oplus\fk$ in terms of the operators $I$ and
$\ad_X$. Specifically, for any vectors $\xi\in{\mathfrak m}$,
$\zeta\in\fk$, we have
\begin{align}\label{eq.4.10}
[Y,\xi] & = [IX,\xi]
\stackrel{\mathrm{(\ref{eq.4.2})}}=
[-X,I\xi]=-\ad_X I\xi,
\\ \label{eq.4.11}
[Y,\zeta] & = [IX,\zeta]
\stackrel{\mathrm{(\ref{eq.4.2})}}=
I[X,\zeta]=I\ad_X\zeta.
\end{align}
Similarly, for $Z=[Y,X]$, using the Jacobi identity
and relations (\ref{eq.4.10}),~(\ref{eq.4.11}) we obtain that
\begin{align}
[Z,\xi] & =[Y,[X,\xi]]-[X,[Y,\xi]]
=I\ad_X^2\xi+ \ad_X^2I\xi, \label{eq.4.12} \\ \noalign{\medskip}
[Z,\zeta] & =[Y,[X,\zeta]]-[X,[Y,\zeta]]
=-2\ad_XI\ad_X \zeta. \label{eq.4.13}
\end{align}

From the definitions of
$Z_{0}$ and $Z_{1}$ in~(\ref{eq.4.1}) and~(\ref{eq.4.3}),
respectively, it follows that
$\ad_{Z_0}|_{{\mathfrak m}_{\varepsilon/{\scriptscriptstyle 2}}}$
$= \ad_{Z_1}|_{{\mathfrak m}_{\varepsilon/{\scriptscriptstyle 2}}}$.
Moreover, since $b_1 - 1 =2(b_0-1)$, then from~(\ref{eq.4.1})
and~(\ref{eq.4.7}) we obtain that
$Z_1-2Z_0=2Z$. In other words,
\begin{equation}\label{eq.4.14}
\ad_{Z_1}|_{{\mathfrak m}_{\varepsilon/{\scriptscriptstyle 2}}}
=I|_{{\mathfrak m}_{\varepsilon/{\scriptscriptstyle 2}}}
\quad\text{and}\quad
Z-\tfrac12 Z_1=-Z_0.
\end{equation}

The operator-functions in \eqref{eq.3.20} are given here by
\begin{alignat}{2}
\textstyle
R_x|_{{\mathfrak m}_{\varepsilon}\oplus\fk_{\varepsilon}}
& =\textstyle\frac{1}{\cosh x}
\mathrm{Id}_{{\mathfrak m}_{\varepsilon}
\oplus\fk_{\varepsilon}},
&
S_x|_{{\mathfrak m}_{\varepsilon}\oplus\fk_{\varepsilon}}
&\textstyle
=\frac{1}{\sinh x}
\ad_X|_{{\mathfrak m}_{\varepsilon}
\oplus\fk_{\varepsilon}}, \label{eq.4.15} \\
\textstyle
R_x|_{{\mathfrak m}_{\varepsilon/{\scriptscriptstyle 2}}
\oplus\fk_{\varepsilon/{\scriptscriptstyle 2}}}
& =\textstyle\frac{1}{\cosh \frac 12 x}
\mathrm{Id}_{{\mathfrak m}_{\varepsilon/{\scriptscriptstyle 2}}
\oplus\fk_{\varepsilon/{\scriptscriptstyle 2}}},
&
\quad
S_x|_{{\mathfrak m}_{\varepsilon/{\scriptscriptstyle 2}}
\oplus\fk_{\varepsilon/{\scriptscriptstyle 2}}}
&\textstyle
= \frac{2}{\sinh \frac 12 x}
\ad_X|_{{\mathfrak m}_{\varepsilon/{\scriptscriptstyle 2}}
\oplus\fk_{\varepsilon/{\scriptscriptstyle 2}}}.
\label{eq.4.16}
\end{alignat}

Put $\xi^1_{\varepsilon}=Y\in {\mathfrak m}_{\varepsilon}$.
With the notation of the previous subsection,
$\zeta^1_{\varepsilon}=Z\in\fk_{\varepsilon}$.
Now we have to verify conditions $(1)\!\!-\!\!(4)$ of
Theorem~\ref{th.3.5} for the vector-function
\begin{equation}\label{eq.4.17}
\mathbf{a}(x)=a^\fa(x)+a^\fk(x)
+a^{\mathfrak m}(x)+z_\fh=
f'(x)X + c_Z\varphi(x)Z +c_Y\psi(x)Y +c_1Z_1,
\end{equation}
where
\[
f\in C^\infty({\mathbb R}^+,{\mathbb R}),\quad
\quad \varphi(x)=\frac{1}{\cosh x},
\quad \psi(x)=\frac{1}{\sinh x},
\quad c_{Y},c_{Z}, c_{1} \in {\mathbb R}.
\]

Consider now the first condition in~(\ref{eq.3.23}). We have
the splitting
${\mathfrak m}^+={\mathfrak m}_{\varepsilon}\oplus{\mathfrak
m}_{\varepsilon/{\scriptscriptstyle 2}}$.
Taking into account that by its definition
$[\fz(\fh),\fg_\fh]=0$
and ${\mathfrak m}_{\varepsilon}=\bbR Y
\subset\fg_\fh$,
using relations~(\ref{eq.4.15}), we can rewrite the
first condition in~(\ref{eq.3.23}) for the vector
$Y=\xi^1_\varepsilon$ as
\begin{equation}\label{eq.4.18}\textstyle
\frac{1}{\cosh x}\cdot R_x
\bigl[c_{Z}\varphi(x)Z,Y\bigr]
+\frac{1}{\sinh x}\cdot S_x
\bigl[c_Z\varphi(x)Z, \ad_X Y\bigr]=0.
\end{equation}
The first term in~(\ref{eq.4.18}) vanishes because
$[Z,Y]=X\in\fa$ and
$R_x(\fa)=0$; the second term vanishes because
$\ad_X Y=-Z$.

Since in our case $\fm_\tw={\mathfrak m}_{\varepsilon/{\scriptscriptstyle 2}}$
and $\fk_\tw={\mathfrak k}_{\varepsilon/{\scriptscriptstyle 2}}$, then
by Remark~\ref{re.3.1}
$[\fg_H\oplus[\fh,\fh],
{\mathfrak m}_{\varepsilon/{\scriptscriptstyle 2}}\oplus
{\mathfrak k}_{\varepsilon/{\scriptscriptstyle 2}}]\subset
{\mathfrak m}_{\varepsilon/{\scriptscriptstyle 2}}\oplus
{\mathfrak k}_{\varepsilon/{\scriptscriptstyle 2}}$.
Let now
$\xi\in{\mathfrak m}_{\varepsilon/{\scriptscriptstyle 2}}$.
Using relations~(\ref{eq.4.12}),
(\ref{eq.4.13}) and~(\ref{eq.4.16}), expression~(\ref{eq.4.14})
and the fact that
$I\xi\in{\mathfrak m}_{\varepsilon/{\scriptscriptstyle 2}}$,
we can rewrite the first condition in~(\ref{eq.3.23}) as
\begin{equation*}\label{eq.4}
\begin{split}
0& \textstyle
=\frac{c_Z\varphi(x)}{\cosh \frac 12 x}\cdot R_x
\bigl[Z,\xi\bigr]
+\frac{2 c_Z\varphi(x)}{\sinh \frac12 x}\cdot S_x
\bigl[Z,\ad_X \xi\bigr]
+\Bigl(\frac{1}{\cosh^2 \frac 12 x}-\frac{1}{\sinh^2 \frac 12 x}
\Bigr)\cdot[c_1Z_1,\xi] \\
& \textstyle
=\frac{c_Z\varphi(x)}{\cosh^2 \frac 12 x}\cdot(I\ad_X^2\xi+\ad_X^2 I\xi)
+\frac{4 c_Z\varphi(x)}{\sinh^2 \frac 12 x}\cdot \ad_X
(- 2\ad_X I\ad_X^2\xi)  -\frac{c_1}{\cosh^2 \frac 12
x\sinh^2 \frac 12 x}\cdot I\xi.
\end{split}
\end{equation*}
Then
\begin{equation*}
\begin{split}
0& \textstyle =\left(\frac{1}{\cosh^2 \frac 12 x}
+\frac{1}{\sinh^2 \frac 12 x}\right)
\frac{-c_Z}{2\cosh x}I\xi
-\frac{c_1}{\cosh^2 \frac 12 x\sinh^2 \frac 12 x}I\xi
=\frac{1}{\cosh^2 \frac 12 x\sinh^2 \frac 12 x}
(-\tfrac12 c_Z-c_1)I\xi,
\end{split}
\end{equation*}
because $\ad_X^2|_{{\mathfrak m}_{\varepsilon/{\scriptscriptstyle 2}}}
=-\tfrac14
\mathrm{Id}_{{\mathfrak m}_{\varepsilon/{\scriptscriptstyle 2}}}$.
Thus $c_1=-\tfrac12 c_Z$.

We can also rewrite the second condition in~(\ref{eq.3.23})
for $\xi\in{\mathfrak m}_{\varepsilon/{\scriptscriptstyle 2}}$ as
\begin{equation*}
\begin{split}
\textstyle
\frac{2}{\sinh\frac 12 x}\cdot R_x
\bigl[c_Y\psi(x)Y,\ad_X \xi\bigr]
&\textstyle
-\frac{1}{\cosh\frac 12 x}\cdot S_x
\bigl[c_Y\psi(x)Y, \xi\bigr]=0.
\end{split}
\end{equation*}
Taking into account
the relations~(\ref{eq.4.10}),~(\ref{eq.4.11}) and~(\ref{eq.4.16})
we obtain that
\begin{equation*}\textstyle
\frac{2c_Y\psi(x)}{\sinh\frac 12 x\cosh\frac 12 x}\cdot
(I\ad_X^2\xi+\ad_X^2 I\xi)
=-\frac{c_Y}{\sinh\frac 12 x\cosh\frac 12 x \sinh x}\cdot I\xi=0.
\end{equation*}
Thus $c_Y=0$ and therefore the component
$a^{\mathfrak m}(x)$ of $\mathbf{a}(x)$ vanishes.
The second condition in~(\ref{eq.3.23}) holds.

Summarizing the results proved above, we obtain that
for the vector-function (\ref{eq.4.17}),
condition $(1)$ of Theorem~\ref{th.3.5}
for $G/K={\mathbb C}{\mathbf P}^n$ $(n\geqslant 2)$ is
equivalent to the conditions
$\{c_Z  \in{\mathbb R},\ c_1=-\tfrac12 c_Z,\ c_Y=0\}$.

Let us describe the $2\times 2$ Hermitian matrix-function
${\mathbf w}_{H}(x)$,
$2=\dim\fa+\dim{\mathfrak m}_{\varepsilon}$,
accord\-ing to condition
$(2)$ of Theorem~\ref{th.3.5}.

It is clear that
$w_{1|1}(x)=2f''(x)$. The function $w_{1|{}_\varepsilon^{1}}(x)$
is determined by the relation
(for $\lambda=\varepsilon$):
$w_{1|{}_\varepsilon^{1}}(x)=2\ri\frac{c_Z}
{\cosh^2 x}-2\frac{c_Y}{\sinh^2 x}$.

The function $w_{{}_\varepsilon^{1}|{}_\varepsilon^{1}}(x)$
is determined by the relation~(\ref{eq.3.24})
for $\zeta^1_\varepsilon =Z$ and $\xi^1_\varepsilon=Y$.
By relations~(\ref{eq.4.8}) and the invariance of the form
$\langle\cdot,\cdot\rangle$,
$$\textstyle
w_{{}_\varepsilon^{1}|{}_\varepsilon^{1}}(x)
=
-\frac{2\ri}{\sinh^2 x}\bigl\langle
\bigl[c_Z\varphi(x)Z+c_1Z_1,Z\bigr],Z\bigr\rangle
-\frac{2}{\cosh x\sinh x}\bigl\langle
\bigl[f'(x)X,Y\bigr],Z\bigr\rangle
=f'(x)\frac{2}{\cosh x\sinh x}.
$$
Hence, we conclude that the entries of ${\mathbf w}_{H}(x)$ are
\begin{equation}\label{eq.4.19}
w_{11}(x)=2f''(x), \quad
w_{1|{}_\varepsilon^{1}}(x)
=2\ri\tfrac{c_Z}{\cosh^2 x},\quad
w_{{}_\varepsilon^{1}|{}_\varepsilon^{1}}(x)
=\tfrac{2f'(x)}{\cosh x\sinh x}.
\end{equation}

Let us describe the Hermitian $s\times s$-matrix
${\mathbf w}_{*}(x)=\bigl(w_{jk}(x)\bigr)$,
$s=\dim{\mathfrak m}_{\varepsilon/{\scriptscriptstyle 2}}=2n-2$,
with entries $w_{jk}(x)= w_{{}_{\varepsilon/{\scriptscriptstyle 2}}^j|
{}_{\varepsilon/{\scriptscriptstyle 2}}^k}(x)$, $j,k= 1, \dotsc, 2n-2$,
determined by relations (\ref{eq.3.24}):
\begin{equation*}
\begin{split}
w_{jk}(x) = & -\textstyle \frac{2\ri}{\sinh^2 \frac{x}{2}}
\big\langle\bigl[c_Z\varphi(x) Z+c_1Z_1,\,
-2\ad_X \xi^j\bigr], -2\ad_X \xi^k\big\rangle \\ & \textstyle
-\frac{2}{\cosh \frac{x}{2} \sinh \frac{x}{2}}
\big\langle\bigl[f'(x)X,\xi^j\bigr], \,-2\ad_X \xi^k \big\rangle,
\end{split}
\end{equation*}
where we put $\xi^j=\xi^j_{\varepsilon/{\scriptscriptstyle 2}}$ to simplify
notation. Taking into account
relations (\ref{eq.4.13}), (\ref{eq.4.10})
and the commutation relation $[\ad_{Z_1},\ad_X]=0$, we obtain that
\begin{equation*}
\begin{split}
w_{jk}(x)
=&\textstyle
-\frac{2\ri}{\sinh^2 {\frac 12}x}
\bigl(-8c_Z\varphi(x) \langle \ad_X I\ad_X^2 \xi^j, \, \ad_X \xi^k\rangle
+4c_1\langle \ad_X I\xi^j,\ad_X \xi^k \rangle\bigr) \\
&\textstyle
-\frac{2}{\cosh {\frac 12}x  \sinh {\frac 12}x }
(-2)f'(x)\langle \ad_X  \xi^j,\ad_X\xi^k\rangle
\\
=&\textstyle
- \frac{\ri}{\sinh^2 {\frac 12}x}
c_Z\Bigl(\frac{1}{\cosh x} - 1 \Bigr)
\langle I\xi^j,\xi^k \rangle
+\frac{1}{\cosh {\frac 12}x  \sinh {\frac 12}x }
f'(x)\langle  \xi^j,\xi^k\rangle.
\end{split}
\end{equation*}

But the orthonormal basis
$\{\xi^j_{\varepsilon/{\scriptscriptstyle 2}}\}_{j=1}^{2n-2}$
is chosen in such a way that
$\xi^{2j}_{\varepsilon/{\scriptscriptstyle 2}}
= I\xi^{2j-1}_{\varepsilon/{\scriptscriptstyle 2}}$.
Thus from the relations above it follows that
the Hermitian matrix ${\mathbf w}_{*}(x)$
is a block-diagonal matrix, where each block is
an Hermitian $2\times 2$-matrix. Each such block
is determined by the pair of vectors
$\bigl(\xi^{2j-1}_{\varepsilon/{\scriptscriptstyle 2}},
\xi^{2j}_{\varepsilon/{\scriptscriptstyle 2}}\big)$,
$j= 1, \dotsc, n-1$, and it is a
$2\times 2$ Hermitian matrix with the entries
\begin{equation}\label{eq.4.20}
\begin{split}
w_{(2j-1)(2j-1)}(x)=w_{(2j)(2j)}(x)&= \textstyle
\frac{f'(x)}{\cosh {\frac 12}x\sinh {\frac 12}x},\\
w_{(2j-1)(2j)}(x) = \textstyle
\frac{\ri c_Z}{\sinh^2 {\frac 12}x} & \textstyle
\Bigl(1 - \frac{1}{\cosh x}\Bigr).\quad
\end{split}
\end{equation}

It is easily checked (calculating determinants of order $2$) that
vector-function~(\ref{eq.4.17})
satisfies conditions $(2)$, $(3)$ and $(4)$ of
Theorem~\ref{th.3.5} if and only if
all the mentioned Hermitian $2\times 2$ matrices are positive-definite
and $\det {\mathbf w}_{H}(x)$ $\cdot
\det {\mathbf w}_{*}(x)=2^{2n-2}\cdot C^n$, $C>0$,
i.e.\ for all $x\in{\mathbb R}^+$ the following relations hold:
\begin{equation}\label{eq.4.21}
\begin{split}
& \textstyle
\mathbf{(a)}\ f''(x)>0,\quad
\mathbf{(b)}\ \frac{f''(x)f'(x)}{\cosh x\sinh x}
-\frac{c_Z^2}{\cosh^4 x}
>0,\\
& \textstyle
\mathbf{(c)}\ f'(x)>0,\quad \hskip3pt
\mathbf{(d)}\ \frac{f'(x)f'(x)} {\cosh^2 {\frac 12}x\sinh^2 {\frac 12}x}
-\frac{c_Z^2}{\sinh^4 {\frac 12}x}
\Bigl(1-\frac{1}{\cosh x}\Bigr)^2>0,
\end{split}
\end{equation}
and
\begin{equation}\label{eq.4.22}\textstyle
\left(\frac{f''(x)f'(x)}{\cosh x\sinh x}
-\frac{c_Z^2}{\cosh^4 x}
\right)\!\!\left(
\frac{f'(x)f'(x)} {\cosh^2 {\frac 12}x\sinh^2 {\frac 12}x}
-\frac{c_Z^2}{\sinh^4 {\frac 12}x}
\Bigl(1-\frac{1}{\cosh x}\Bigr)^2\right)^{n-1}\!\!=2^{2n-2}C^n.
\end{equation}
However, there exists an exact general solution of equation~(\ref{eq.4.22}).
Indeed, taking into account some well-known identities for
the functions $\cosh x$ and $\sinh x$, and using the substitution
$g_1(x)=(f'(x))^2$ one can rewrite~(\ref{eq.4.22}) as
\begin{equation*}
g_1'(x)= \frac{2c_Z^2\sinh x}{\cosh^3 x} + 2^{2n-1} C^n
\cosh x\sinh x\left(\frac{\cosh^2 x\sinh^2 x}{4g_1(x)
\cosh^2 x - 4c_Z^2\sinh^2 x}
\right)^{n-1}.
\end{equation*}
Next, using the substitution
$g_2(x)=g_1(x) \cosh^2 x-c_Z^2\sinh^2 x$ we obtain the
Bernoulli equation
\begin{equation*}
g'_2(x)=
\frac{2\sinh x}{\cosh x}\cdot g_2(x)+
2 C^n\cosh^3 x\sinh x\left(\frac{\cosh^2 x\sinh^2 x}
{g_2(x)}\right)^{n-1}
\end{equation*}
with solutions
$g_2(x)=\cosh^2 x\left(C^n\sinh^{2n}x+C_1\right)^{1/n}$, $C_1\geqslant 0$,
on the whole semi-axis,  i.e.\ we obtain that
\begin{equation}\label{eq.4.23}
f'(x)=\sqrt{(C^n\sinh^{2n}x+C_1)^{1/n}
+c_Z^2\sinh^2 x\cosh^{-2} x}.
\end{equation}
and therefore
\begin{equation}\label{eq.4.24}
\begin{split}
f''(x)&=(f'(x))^{-1}\bigl(
C^n\cdot(C^n\sinh^{2n}x+C_1)^{\frac{1-n}{n}}\cdot \sinh^{2n-1} x\cosh x \\
&\quad\;+c_Z^2(\tanh x-\tanh^3 x) \bigr).
\end{split}
\end{equation}
For these functions on the whole semi-axis
relations~(\ref{eq.4.21}c) and~(\ref{eq.4.21}a)
hold since $\sinh x>0$ and
$\tanh x>\tanh^3 x$ on this set ($0<\tanh x<1$); also
(\ref{eq.4.21}b) hold, because $\tanh x-\tanh^3 x
=\sinh x\cosh^{-3} x$; and~(\ref{eq.4.21}d) hold,  as
$\sinh^2 x\cosh^{-2}x=
\frac{\cosh^2 {\frac 12}x}{\sinh^2 {\frac 12}x}
\Bigl(1-\frac{1}{\cosh x}\Bigr)^2$.

The form
$\widetilde{\omega} = {\rm d}\widetilde{\theta}^{\mathbf a}$
on $G\times\bbR^+$ determines a unique form $\omega$ on
$G/H\times \bbR^+=G/H\times W^+$  such that
$\widetilde{\omega}=(\pi_H \times \mathrm{id})^\ast \omega$
(see Corollary~\ref{co.3.6}).
Let us study when the form $\omega$
on $G/H\times W^+\cong T^+(G/K)$ admits an
smooth extension to the whole tangent space $T(G/K)$.
To this end we will find the expression of the form
$\omega^R=((f^+)^{-1})^*\omega$ on the space
$G\times_K{\mathfrak m}^R\cong T^+(G/K)$, where, recall,
$f^+\!\colon G/H\times {\mathbb R}^+\to G\times_K{\mathfrak m}^R$ is a
$G$-equivariant diffeomorphism.
However, by the commutativity of diagram~\eqref{eq.3.16}
there exists a unique form $\widetilde\omega^R$ on $G\times{\mathfrak m}^R$
such that
\begin{equation}\label{eq.4.25}
\widetilde\omega^R=\pi^*\omega^R
\quad\text{and}\quad
\widetilde{\omega}=\mathrm{id}^*\widetilde\omega^R.
\end{equation}

Thus it is sufficient to calculate the form $\widetilde\omega^R$
on the space $G\times{\mathfrak m}^R$, because the form $\omega^{R}$
on the space $G\times_{K}{\mathfrak m}^{R}\cong T^{+}(G/K)$
may be extended (in a unique way if the extension
does exist) to the whole tangent space $T(G/K)$ if and
only if the form $\widetilde{\omega}^{R}$  is extendable
(admits an extension to the whole space $G\times {\mathfrak m}$).

By the second expression in (\ref{eq.4.25}),
\begin{equation}\label{eq.4.26}
\widetilde{\omega}^R_{(g,xX)}\bigl(\bigl(\xi_1^l(g),t_1 X\big),
\bigl(\xi_2^l(g),t_2 X\big)\big)=
\widetilde{\omega}_{(g,x)}\bigl( \bigl(\xi_1^l(g),
t_1\tfrac{\partial}{\partial x}\big),
\bigl(\xi_2^l(g),t_2\tfrac{\partial}{\partial x} \big)\big).
\end{equation}
To describe $\widetilde{\omega}^{R}$ we consider again
the two cases (1) and (2):

{\bf(1)} $G/K\in\{ \mathbb{S}^n(n \geqslant 3),
\, {\mathbb H}{\mathbf P}^n(n\geqslant 1),
\,{\mathbb C\mathrm a}{\mathbf P}^2 \}$. Since ${\mathbf a}(x)
= f'(x)X$, by~\eqref{eq.3.22} at the point
$(g,xX)\in G\times W^+$ and from~(\ref{eq.4.26}) we have that
\[
\widetilde{\omega}^{R}_{(g,xX)} \bigl((\xi_1^l(g),t_1X),
(\xi_2^l(g),t_2 X)\bigr)= -\bigl\langle f'(x)X,[\xi_1,\xi_2]
\bigr\rangle +f''(x) \bigl( t_1\bigl\langle  X, \xi_2 \bigr\rangle
-t_2\bigl\langle X, \xi_1 \bigr\rangle \big),
\]
where $\xi_1,\xi_2\in\fg=T_eG$ and
$t_1,t_2\in{\mathbb R}$. Consider on  the whole tangent space
$T_{(g,w)}(G\times{\mathfrak m}^R)$
($w\in {\mathfrak m}^{R} = {\mathfrak m}\setminus\{0\})$, the
bilinear form $\Delta$ given by
\begin{equation}\label{eq.4.27}
\begin{split}
&\Delta_{(g,w)}\bigl((\xi_1^l(g),u_1),(\xi_2^l(g),u_2)\bigr)
=-\bigl\langle \tfrac{f'(\bfr)}{\bfr } w,\, [\xi_1,\xi_2] \bigr\rangle \\
&\qquad +\tfrac{f'(\bfr)}{\bfr } \bigl(\langle u_1, \xi_2 \rangle
-\langle u_2,\xi_1\rangle \bigr)
+\tfrac{1}{\bfr } \left(\tfrac{f'(\bfr)}{\bfr }\right)'
\bigl(\langle u_1,w \rangle \langle w, \xi_2 \rangle
-\langle u_2,w \rangle \langle w, \xi_1 \rangle\bigr),
\end{split}
\end{equation}
where $\xi_1,\xi_2\in\fg=T_eG$,
$u_1,u_2\in {\mathfrak m}=T_w{\mathfrak m}^R$.
Here $\mathbf{r}=\mathbf{r}(w)$, $\mathbf{r}^2(w)\eqdef
{\langle w, w\rangle}$ ($\mathbf{r}(xX)$ $= x$).
It is clear that this form is skew-symmetric.

Since
$\tfrac{1}{\bfr }
\left(\tfrac{f'(\bfr)}{\bfr }\right)'
=\tfrac{\bfr f''(\bfr)-f'(\bfr)}{\bfr ^3}$,
it is easy to verify that
$$
\Delta_{(g,xX)}\bigl((\xi_1^l(g),t_1X),(\xi_2^l(g),t_2 X)\big)=
\widetilde{\omega}^R_{(g,xX)}\bigl((\xi_1^l(g),t_1 X),(\xi_2^l(g),t_2 X)\big),
$$
i.e.\ the restrictions of
$\widetilde{\omega}^{R}$ and $\Delta$ to
$G\times W^{+}\subset G\times\fm^R$ coincide.
Now to prove that the differential
forms $\widetilde{\omega}^R$ and
$\Delta$ coincide on the whole tangent bundle
$T(G\times{\mathfrak m}^R)$ it is sufficient to show that
the form $\Delta$ is left $G$-invariant, right
$K$-invariant and its kernel contains (and therefore
coincides with) the subbundle
$\sK=\ker \pi_{*}$.

Since for each $k\in K$ the scalar product $\langle\cdot,
\cdot\rangle$ is $\Ad_{k}$-invariant and $\Ad_{k}$
is an automorphism of $\fg$, the following relations hold:
\begin{equation}
\label{eq.4.28}
\begin{split}
\Delta_{(g,w)}((\xi_1^l(g),u_1),\, &(\xi_2^l(g),u_2))
=\Delta_{(e,w)} ((\xi_1,u_1),(\xi_2,u_2)) \\
& = \Delta_{(e,\Ad _k w)}(
(\Ad_k \xi_1,\Ad _k u_1),(\Ad_k \xi_2,\Ad_k u_2)).
\end{split}
\end{equation}
Hence, $\Delta$ is left $G$-invariant and right $K$-invariant.

The kernel
${\sK}\subset T(G\times {\mathfrak m})$ of the tangent map
$\pi_{*} \colon T(G\times\fm)\to T(G\times_K\fm)$
is generated by the (left)
$G$-invariant vector fields $\zeta^{L}$,
$\zeta\in \fk$~(\ref{eq.3.17}) on
$G\times{\mathfrak m}$. Then, since
${\mathfrak m}\bot\fk$ and
$\langle w,[w,\zeta]\rangle = 0$, we obtain
\begin{equation*}
\begin{split}
& \Delta_{(g,w)}\bigl((\xi_1^l(g),u_1),\zeta^{L}{(g,w)}\big)=
-\bigl\langle \tfrac{f'(\bfr)}{\bfr } w
,\, [\xi_1,\zeta] \bigr\rangle  +\tfrac{f'(\bfr)}{\bfr }
\bigl(\langle u_1, \zeta \rangle
-\langle [w,\zeta],\xi_1\rangle \bigr) \\
&+\tfrac{1}{\bfr }\left(\tfrac{f'(\bfr)}{\bfr }\right)'
\bigl(\langle u_1,w \rangle \langle w, \zeta \rangle
-\langle [w,\zeta],w \rangle \langle w, \xi_1 \rangle\bigr)
=  -\bigl\langle \tfrac{f'(\bfr)}{\bfr } w,[\xi_1,\zeta] \bigr\rangle
 -\bigl\langle \tfrac{f'(\bfr)}{\bfr }[w,\zeta], \xi_1 \bigr\rangle
 =0.
\end{split}
\end{equation*}
This means that $\sK\subset \ker\Delta$.
Thus $\widetilde{\omega}^R=\Delta$ on $G\times{\mathfrak m}^R$
($\sK =\ker\Delta$ because the form
$\omega$ is nondegenerate).

Expression~(\ref{eq.4.27})
determines a smooth $2$-form on the whole tangent bundle
$T(G\times{\mathfrak m})$ if and only if
$\lim_{x\to 0}f'(x) = 0$, that is,
$C_{1} = 0$. Indeed, if $C_1>0$ it is easy to verify that
$\lim_{x\to 0}f'(x) =C_1^{1/m}$ and $\lim_{x\to 0}f''(x) =0$,
where $m=\dim\fm=m_\varepsilon+m_{\varepsilon/{\scriptscriptstyle 2}}+1$.
Therefore, by~(\ref{eq.4.27}), $\lim_{x\to 0} \Delta_{(e,xX)}
\bigl((\xi_1,u_1),(\xi_2,u_2)\big)=\infty$ for some
vectors $\xi_1,\xi_2\in\fg$, $u_1,u_2\in\fm$ such that
$$
\lim_{x\to 0}\tfrac{f'(x)}{x} \bigl(\langle u_1, \xi_2 \rangle
-\langle u_2,\xi_1\rangle
-\langle u_1,X \rangle \langle X, \xi_2 \rangle
+\langle u_2,X \rangle \langle X, \xi_1 \rangle
\bigr)=\infty.
$$
Let $C_1=0$. Since $\frac{\sinh x}{x}>1$ for $x>0$,
there exists an even real analytic function on the whole axis,
$\psi_2(x)$, such that
\begin{equation}\label{eq.4.29}
f'(x)=x\Bigl(\tfrac{2^{m_\varepsilon}C}{m}+x^2\psi_2(x)\Bigr)^{1/m},
\quad
\tfrac{2^{m_\varepsilon}C}{m}+x^2\psi_2(x)>0,\quad\forall x>0.
\end{equation}
In this case expression~(\ref{eq.4.27})
determines a smooth 2-form on the whole space $G\times \fm$.

We will
denote this form (extension) on
$G\times \fm$ by
$\widetilde{\omega}^{R}_{0}$. There exists a unique
$2$-form $\omega^{R}_{0}$ on
$G\times_{K}\fm\cong T(G/K)$ such that
$\widetilde{\omega}^{R}_{0} = \pi^{*}\omega^{R}_{0}$.
The forms $\omega^R_{0}$ and $\omega^{R}$
coincide, by construction, on the open submanifold
$G\times_{K}\fm^{R}\cong T^+(G/K)$, that is,
$\omega^{R}_{0}$ is a smooth extension of
$\omega^{R}$.

Now we will prove, applying Corollary
\ref{co.3.4}, that this extension is the K\"ahler form
of the metric ${\mathbf g}_0$ on the whole tangent bundle
$T(G/K)$. Indeed,  by~(\ref{eq.4.6}) and~(\ref{eq.4.29})
for $C_1=0$,
$$
\lim_{x\to 0} w_{1|1}(x)=
2\left(\frac{2^{m_\varepsilon}C}{m}\right)^{1/m},
\quad
\lim_{x\to 0}w_{{}_\varepsilon^j|{}_\varepsilon^j}(x)=
\lim_{x\to 0}w_{{}_{\varepsilon/{\scriptscriptstyle 2}}^k|
{}_{\varepsilon/{\scriptscriptstyle 2}}^k}(x)=
2\left(\frac{2^{m_\varepsilon}C}{m}\right)^{1/m},
$$
that is, the corresponding limit diagonal Hermitian matrices
$\lim_{x\to 0}{\mathbf w}_H(x)$ and
$\lim_{x\to 0}{\mathbf w}_\tw(x)$ are positive-definite.
Thus by Corollary~\ref{co.3.4},
$\omega^R_0$ is the K\"ahler form
of the metric ${\mathbf g}_0$ (the extension of ${\mathbf g}$)
on $G\times_K\fm\cong T(G/K)$.
\medskip

{\bf(2)} $G/K ={\mathbb C}{\mathbf P}^{n}$
$(n\geq 2)$. In this case the vector-function
${\mathbf a}$ takes the form
\[
{\mathbf a}(x)=
f'(x)X + c_Z\varphi(x)Z+c_1Z_1
= f'(x)X + c_{Z}\bigl(\varphi(x) -1\big)Z -c_{Z}Z_{0},
\]
because $Z_{1} = 2(Z + Z_{0})$ and $c_1=-\tfrac12 c_Z$.
Here $f'(x)$ is given in (\ref{eq.4.5}),
$\varphi(x) = \frac{1}{\cosh x}$ and
$c_{Z}$ is an arbitrary real number.
Then, from (\ref{eq.3.22}), we have
\begin{equation}\label{eq.4.30}
\begin{split}
&\widetilde{\omega}^{R}_{(g,xX)} \bigl((\xi_1^l(g),t_1X),\,
(\xi_2^l(g),t_2 X)\bigr)= -\bigl\langle f'(x)X+c_{Z}\bigl(\varphi(x)-1\big)Z
-c_{Z}Z_0,[\xi_1,\xi_2] \bigr\rangle \\
&\hskip20mm +f''(x) \bigl( t_1\bigl\langle  X, \xi_2 \bigr\rangle
-t_2\bigl\langle X, \xi_1 \bigr\rangle \bigr)+ c_{Z}\varphi'(x)
\bigl( t_1\bigl\langle Z,
\xi_2 \bigr\rangle -t_2\bigl\langle Z,
\xi_1 \bigr\rangle \bigr),
\end{split}
\end{equation}
where $\xi_1,\xi_2\in\fg=T_eG$ and
$t_1,t_2\in{\mathbb R}$.

Consider on  the whole tangent space
$T_{(g,w)}(G\times\fm^R)$ ($w\ne 0$), the
bilinear form $\Delta$:
\begin{equation}\label{eq.4.31}
\begin{split}
&\Delta_{(g,w)} \bigl((\xi_1^l(g),u_1),(\xi_2^l(g),u_2)\bigr)
= -\bigl\langle \tfrac{f'(\bfr)}{\bfr } w
+\tfrac{c_Z}{\bfr ^{2}}\bigl(\varphi(\bfr) -1\big)[Iw,w]
-c_Z Z_0,[\xi_1,\xi_2] \bigr\rangle\\
&\qquad+ \tfrac{f'(\bfr)}{\bfr }\bigl(\langle u_1,\xi_2\rangle
-\langle u_2,\xi_1\rangle\bigr)
+\tfrac{1}{\bfr }\Bigl(\tfrac{f'(\bfr)}{\bfr }\Bigr)'
\bigl(\langle u_1,w\rangle\langle w,\xi_2 \rangle
-\langle u_2,w\rangle\langle w,\xi_1 \rangle\bigr)\\
&\qquad+c_{Z} \Bigl(\tfrac{\bfr  \varphi'(\bfr)
- 2(\varphi(\bfr)-1)}{\bfr ^{4}}\Bigr)
\bigl(
\langle u_1,w \rangle\langle [Iw,w],\xi_2 \rangle
-\langle u_2,w \rangle\langle [Iw,w],\xi_1 \rangle\big)\\
&\qquad+c_{Z} \Bigl(\tfrac{\bfr  \varphi'(\bfr)
- 2(\varphi(\bfr)-1)}{\bfr ^{4}}\Bigr)
\bigl(\langle u_1,Iw \rangle\langle w,u_2 \rangle
-\langle u_2,Iw \rangle\langle w,u_1 \rangle\bigr)\\
&\qquad+c_Z\tfrac{2(\varphi(\bfr) -1)}{\bfr ^{2}}
\Bigl(\langle[Iu_1,w],\xi_2 \rangle
-\langle[Iu_2,w],\xi_1\rangle
+\langle u_1,Iu_2\rangle\Bigr),
\end{split}
\end{equation}
where $\xi_1,\xi_2\in\fg=T_eG$,
$u_1,u_2\in \fm=T_w\fm^R$. It is clear that this form is
skew-symmetric. From the expression of
$\widetilde{\omega}^{R}$ at the point
$(g,xX)\in G\times W^{+}$ given in~(\ref{eq.4.30}) and taking
into account that $[IX,X]=Z$ and
$\langle X,IX \rangle=0$, it is easy to verify that
$$
\Delta_{(g,xX)}\bigl((\xi_1^l(g),t_1X),(\xi_2^l(g),t_2 X)\big)=
\widetilde{\omega}^R_{(g,xX)}\bigl((\xi_1^l(g),t_1 X),(\xi_2^l(g),t_2 X)\big).
$$

Since for each $k\in K$ the scalar product $\langle \cdot,\cdot \rangle$
is $\Ad_k$-invariant,  $\Ad_k$ is an automorphism of $\fg$  and
$\Ad_k(Z_0)=Z_0$, $\Ad_k I=I \Ad_k$, relations~\eqref{eq.4.28}
hold now for this $\Delta$, that is, $\Delta$
is left $G$-invariant and right $K$-invariant.
We now prove that $\ker \Delta \supset \sK$.
By~(\ref{eq.3.17}), since the form $\Delta$
is left $G$-invariant, right $K$-invariant
and $\Ad(K)({\mathbb R}X)=\fm$,
it is sufficient to show that the vectors $(\zeta,x[X,\zeta])$,
$\zeta\in\fk$, belong to the kernel of $\Delta_{(e,xX)}$.
Indeed, using the fact that $[Z_0,\fk]=0$, $\fk\bot \fm$  and
$\langle\cdot,\cdot\rangle$ is $\Ad(G)$-invariant,
we have that
\begin{equation*}
\begin{split}
\Delta_{(e,xX)}\bigl((\xi_1,u_1),(\zeta,x[X,\zeta])\bigr)
= c_{Z}\bigl(\varphi(x) -1\bigr)\bigl\langle -[IX,X],
[\xi_1,\zeta] \bigr\rangle \qquad\qquad \\
 \qquad\qquad +\tfrac{2}{x}c_Z\big(\varphi(x) -1\bigr)
\cdot\Bigl(\langle[Iu_1,X],\zeta \rangle
-\langle[I[xX,\zeta],X],\xi_1\rangle
+\langle u_1,I[X,\zeta]\rangle\Bigr).
\end{split}
\end{equation*}
This expression vanishes because the endomorphism
$I$ on $\fm$ is skew-symmetric and
\begin{multline*}
-\langle [IX,X],[\xi_1,\zeta]\rangle
-2\langle[I[X,\zeta],X],\xi_1\rangle
\stackrel{(\ref{eq.4.8})}  =
\langle [Z,\zeta],\xi_1\rangle
+2\langle[X,I[X,\zeta]],\xi_1\rangle \\
\stackrel{(\ref{eq.4.13})}=
\langle -2\ad_XI\ad_X \zeta,\xi_1\rangle
+2\langle[X,I[X,\zeta]],\xi_1\rangle=0.
\end{multline*}
Thus the differential forms
$\widetilde{\omega}^R$ and
$\Delta$ coincide on the whole tangent bundle
$T(G\times\fm^R)$.

Our expression~(\ref{eq.4.31}) of $\Delta$
determines a smooth $2$-form on the whole tangent bundle
$T(G\times{\mathfrak m})$ if and only if
$\lim_{x\to 0}f'(x) = 0$, that is,
$C_{1} = 0$. Indeed, if $C_1>0$ it is easy to verify that
$\lim_{x\to 0}f'(x) =C_1^{1/2n}$ and $\lim_{x\to 0}f''(x) =0$.
Therefore by~(\ref{eq.4.31}), $\lim_{x\to 0} \Delta_{(e,xX)}
\bigl((\xi_1,u_1)$, $(\xi_2,u_2)\big)=\infty$ for some
vectors $\xi_1,\xi_2\in\fg$, $u_1,u_2\in\fm$ such that
$$
\lim_{x\to 0} \tfrac{f'(x)}{x} \bigl(\langle u_1, \xi_2 \rangle
-\langle u_2,\xi_1\rangle
-\langle u_1,X \rangle \langle X, \xi_2 \rangle
+\langle u_2,X \rangle \langle X, \xi_1 \rangle
\bigr)=\infty.
$$
Let $C_1=0$. In this case, the expression for the function
$f'(x)$ in~(\ref{eq.4.5}) is independent of
$n$ and there exists an even real analytic function on the whole axis,
$\varphi_2(x)$, such that
\begin{equation}\label{eq.4.32}
f'(x)=x\bigl(C+c^2_Z+x^2\varphi_2(x)\bigr)^{1/2},
\quad
C+c^2_Z+x^2\varphi_2(x)>0,\ \forall x>0.
\end{equation}
Hence by~(\ref{eq.4.5}) the functions
$\tfrac{f'(x)}{x}$ and
$\tfrac{1}{x}\Big(\tfrac{f'(x)}{x}\Big)'$
are even real analytic functions on the whole axis. Also
taking into account that
$\varphi(x)=\tfrac{1}{\cosh x}$ and
$\varphi'(x)=-\varphi(x)\tanh x$ and
$\tanh' x=\varphi^2(x)$ we obtain that
$\varphi(x)=1-\tfrac12 x^2+\tfrac{5}{24}x^4+\varphi_6(x)x^6$,
where $\varphi_6(x)$ is an even real analytic function on the whole
axis ${\mathbb R}$. Therefore the functions
$\tfrac{\varphi(x)-1}{x^2}$ and
$\tfrac{x\varphi'(x)-2(\varphi(x)-1)}{x^4}=\tfrac{5}{12}+
4\varphi_6(x)x^2+\varphi'_6(x)x^3$
are even real analytic functions defined on the whole axis. Therefore
the expression~(\ref{eq.4.31}) determines a smooth 2-form
on the whole tangent space
$T(G/K)$. We will denote, as in the previous cases,
this form (extension) on
$T(G\times \fm)$ by $\widetilde{\omega}^R_0$.

By continuity the form
$\widetilde{\omega}^R_0$ is closed, left $G$-invariant, right $K$-invariant
and ${\mathscr K}\subset{\rm ker}\;\widetilde{\omega}^R_0$.
It is clear that there exists a unique (closed) 2-form
$\omega^R_0$ on $T(G/K)$ such that $\widetilde{\omega}^R_0=\pi^*\omega^R_0$.

Now we will prove, applying Corollary
\ref{co.3.4}, that this extension is the K\"ahler form
of the metric ${\mathbf g}_0$ on the whole tangent bundle
$T(G/K)$. Indeed, the entries of ${\mathbf w}_{H}(x)$ are
determined by expressions~(\ref{eq.4.19}) and therefore
by~(\ref{eq.4.32}),
\begin{equation*}
\lim_{x\to 0} w_{1|1}(x)=2\sqrt{C+c^2_Z},
\quad
\lim_{x\to 0} w_{1|{}_\varepsilon^{1}}(x)
=2\ri\, c_Z,
\quad
\lim_{x\to 0} w_{{}_\varepsilon^{1}|{}_\varepsilon^{1}}(x)
=2\sqrt{C+c^2_Z}.
\end{equation*}
Also from relations~(\ref{eq.4.20}) it follows that
for the block-diagonal Hermitian matrix ${\mathbf w}_{*}(x)$
for each its $2\times 2$ block we have
\begin{equation*}
\begin{split}
\quad \lim_{x\to 0} w_{(2j-1)(2j-1)} & (x)
= \lim_{x\to 0} w_{(2j)(2j)}(x) \textstyle
=2\sqrt{C+c^2_Z}, \\
& \lim_{x\to 0} w_{(2j-1)(2j)}(x)
= 2\ri\, c_Z.
\end{split}
\end{equation*}
It is easy to check
that the corresponding limit diagonal Hermitian matrices
$\lim_{x\to 0}{\mathbf w}_H(x)$ and
$\lim_{x\to 0}{\mathbf w}_\tw(x)$ are positive-definite.
Thus by Corollary~\ref{co.3.4},
$\omega^R_0$ is the K\"ahler form
of the metric ${\mathbf g}_0$ (the extension of ${\mathbf g}$)
on $G\times_K\fm\cong T(G/K)$.
\medskip

\smallskip

Let us prove that the
metric ${\mathbf g}_0$ determined by the form $\omega^R_0$
on the whole tangent bundle $T(G/K)\cong G\times_K\fm$ is complete.

First of all suppose that
$\omega^R_0$ is determined by the vector-function
${\mathbf a}(x)=f'(x)X$. By Corollary~\ref{co.3.7}, such a metric
admits a $G$-invariant potential function $2f(r)$ on
$T(G/K)\setminus G/K$, where
$r$ is the norm function determined by a
$G$-invariant metric on $G/K$. Since in our cases
$f(x)$ is the restriction of an even smooth function on the
whole axis $\bbR$, there exist a smooth extension of
$2f(r)$ to the whole tangent bundle $T(G/K)$.
By continuity, this extension is a potential function on
$T(G/K)$. Now, Stenzel described all
$G$-invariant K\"ahler structures
$(\omega, J^K_c)$ on $T(G/K)$, where
$G/K$ is a compact symmetric space of rank one admitting a
$G$-invariant potential function~\cite{St}.
Thus the set of metrics $c{\mathbf g}_0$, $c>0$, coincides with
Stenzel's set of metrics. The completeness of these metrics is proved
in Stenzel's paper~\cite{St} (see also another proof of this
fact in Mykytyuk~\cite{My}).

Let us prove that the
metric ${\mathbf g}_0$ determined by the form $\omega^R_0$
on the whole tangent bundle $T(G/K)\cong G\times_K\fm$ is complete
if $G/K={\mathbb C}{\mathbf P}^n$, $n\geqslant 2$.
To this end, consider again its
description~\eqref{eq.4.30} on the space
$G/H\times \bbR^+$ ($G=\mathrm{SU}(n{+}1)$,
$H\cong\mathrm{U}(1){\times}\mathrm{SU}(n-1)$).
For our aim it is sufficient to calculate the
distance $\mathrm{dist}(b,c)$ between the compact subsets
$G/H\times\{b\}$ and $G/H\times\{c\}$, where
$\mathrm{dist}(b,c)= \inf \{d(p_b,p_c), p_b\in G/H\times\{b\},
p_c\in G/H\times\{c\}\}$.
Since the sets
$G/H\times\{x\}$ are compact, it is clear that the metric
${\mathbf g}_0$ is complete if and only if for some $b>0$
one has $\lim_{c\to\infty}\mathrm{dist}(b,c)=\infty$.

To calculate the function $\mathrm{dist}(b,c)$
note that the tangent bundle $T(G/K)\cong G\times_ K\fm$
is a cohomogeneity-one manifold, i.e.\ the orbits of the
action of the Lie group $G$ have codimension one.
We will use only one fundamental fact
on the structure of these manifolds~\cite{Al}:
A unit smooth vector field $U$ on a $G$-invariant
domain $D\subset T(G/K)$ which is
${\mathbf g}_0$-orthogonal to each $G$-orbit in $D$
is a geodesic vector field,
i.e.\ its integral curves are geodesics of  the metric
${\mathbf g}_0$.

We now describe such a vector field $U$ on the domain
$G\times \bbR^+\cong T(G/K)\setminus G/K$. Put
\begin{equation}\label{eq.4.33}
f_U(x)=
\left(\frac{f'(x)}{f'(x)f''(x)
+c_Z^2\varphi(x)\varphi'(x)}\right)^{1/2}, \quad x\in\bbR^+,
\end{equation}
where, recall, $\varphi(x)=\tfrac{1}{\cosh x}$.
\begin{lemma}\label{le.4.2}
Such a unit vector field $U$ on $G/H\times \bbR^+$ is $G$-invariant
and at the point $(o,x)$, $o=\{H\}$, $x\in \bbR^+$,
is determined by the expression
\begin{equation}\label{eq.4.34}
U(o,x) = f_U(x)\cdot
\left(-\frac{c_Z\varphi'(x)}{f'(x)} Y,
\,\frac{\partial}{\partial x}\right).
\end{equation}
For the coordinate function
$x$ on $G/H\times \bbR^+$ the following inequality holds
\begin{equation}\label{eq.4.35}
\big|\rd x_{(o,x)}\bigl(\xi,\,
t\tfrac{\partial}{\partial x}\big)\big|\leqslant
f_U(x)\cdot
\big\|(\xi,t\tfrac{\partial}{\partial x})\big\|_{(o,x)},
\end{equation}
where $\bigl(\xi,\,t\tfrac{\partial}{\partial x}\big)\in
T_{(o,x)}(G/H\times \bbR^+)=(\fm\oplus\fk^+)\times\bbR$ and
$\|\cdot\|$ is the norm given by the metric~${\mathbf g}$.
\end{lemma}
\noindent {\it Proof} (of Lemma)
Since the vector field $U$ is unique (up to sign),
it is sufficient to verify that
each vector $U(o,x)$ in~(\ref{eq.4.34}) is
${\mathbf g}$-orthogonal to the $G$-orbit through
$(o,x)$, i.e. to the subspace
$V(o,x) \subset T_{(o,x)}(G/H$ $\times \bbR^+)$
generated by the vectors
$(\xi,0)$, $\xi\in\fm\oplus\fk^+$, and that
$\|U(o,x)\|=1$.

Using expression~\eqref{eq.4.30} for the form
$\tio^R$ we obtain the following expression for the form
$\omega$ at $(o,x)$:
\begin{equation}\label{eq.4.36}
\begin{split}
&\omega_{(o,x)}\Bigl(\bigl(\xi_1,t_1\tfrac{\partial}{\partial x}\bigr),\,
\bigl(\xi_2,t_2\tfrac{\partial}{\partial x}\bigr)\Bigr)
= -\bigl\langle f'(x)X+c_{Z}\bigl(\varphi(x)-1\big)Z
-c_{Z}Z_0,[\xi_1,\xi_2] \bigr\rangle \\
&\hskip20mm +f''(x) \bigl( t_1\bigl\langle  X, \xi_2 \bigr\rangle
-t_2\bigl\langle X, \xi_1 \bigr\rangle \bigr)+ c_{Z}\varphi'(x)
\bigl( t_1\bigl\langle Z,
\xi_2 \bigr\rangle -t_2\bigl\langle Z,
\xi_1 \bigr\rangle \bigr),
\end{split}
\end{equation}
where $\xi_1,\xi_2\in\fm\oplus\fk^+$, $t_1,t_2\in\bbR$.

Fix a point $(o,x)\in G/H\times\bbR^+$
and consider a tangent vector $\widetilde Y=\left(b Y,\,
\tfrac{\partial}{\partial x}\right)$, $b\in\bbR$,
at $(o,x)$. This vector is
${\mathbf g}$-orthogonal to
$V(o,x)$ if and only if this vector is
$\omega$-orthogonal to the subspace $J^K_c(V(o$, $ x))$ generated by
the vectors $(\xi_1,t_1 \tfrac{\partial}{\partial x})$,
$\xi_1\in\fm^+\oplus\fk^+$ $\bigl(\langle \xi_1, X\rangle=0\bigr)$,
$t_1\in\bbR$, because by~\eqref{eq.3.18},
\begin{equation}\label{eq.4.37}
J_c^K(o,x)(X,0)= \bigl(0,\,\tfrac{\partial}{\partial x}\big),\quad
J_c^K(o,x)(Y,\,0)= \bigl(-\tfrac{\cosh x}{\sinh x}Z,0 \bigr)
\end{equation}
and $J_c^K(o,x)(\xi^{j}_{\varepsilon/{\scriptscriptstyle 2}},
\,0)=(-\tfrac{\cosh x}{\sinh x}
\zeta^{j}_{\varepsilon/{\scriptscriptstyle 2}},0)$,
$j=1,\ldots,2(n-1)$. By~(\ref{eq.4.36}) for any
$\xi_1\in\fm^+\oplus\fk^+$,
$t_1\in\bbR$ we obtain
\begin{equation}\label{eq.4.38}
\begin{split}
\omega_{(o,x)}\bigl(\bigl(\xi_1,t_1\tfrac{\partial}{\partial x}\big),
\bigl(b Y,\tfrac{\partial}{\partial x}\big)\big)
&=b\bigl\langle [ f'(x)X+c_{Z}\bigl(\varphi(x)-1\big)Z
-c_{Z}Z_0, Y], \xi_1 \bigr\rangle \\
+f''(x) \bigl( t_1\bigl\langle  X, b Y \bigr\rangle
&-\bigl\langle X, \xi_1 \bigr\rangle \bigr)+ c_{Z}\varphi'(x)
\bigl( t_1\bigl\langle Z,
b Y \bigr\rangle -\bigl\langle Z,
\xi_1 \bigr\rangle \bigr) \\
&=-(bf'(x)+c_Z\varphi'(x))\langle Z, \xi_1 \bigr\rangle,
\end{split}
\end{equation}
because $\xi_1\bot X$ and $[Z_0,Y]=IY=-X$
(see also relations~\eqref{eq.4.8}). By~(\ref{eq.4.38})
the vector $U(o,x)=f_U(x)\widetilde Y$ with
$b=-c_Z\varphi'(x)/f'(x)$ is ${\mathbf g}$-orthogonal
to the subspace $V(o,x)$.

But by~(\ref{eq.4.37}), $J_c^K(o,x)
\left(-\tfrac{c_Z\varphi'(x)}{f'(x)} Y,\,
\tfrac{\partial}{\partial x}\right)=
\left(-\tfrac{c_Z\varphi(x)}{f'(x)} Z-X,\,0\right)$
because $\varphi'(x)=-\varphi(x)\tanh x$.
Taking into account relations~\eqref{eq.4.36}, (\ref{eq.4.31}),
(\ref{eq.4.8}) and the fact that $\langle Z_0,Z\rangle=-1$
we obtain that  $\displaystyle \omega(J_c^K(U),U)=
f^2_U \left(f''+\tfrac{c_Z^2\varphi'\varphi}{f'}\right)\equiv 1$,
i.e. $\|U\|\equiv 1$.

To prove the inequality in the statement it is sufficient to find
the Hamiltonian vector field $\mathtt{H^x}$ of the function $x$.
This vector field is $G$-invariant as so are the form $\omega$ and the
function $x$. Let us show that
$\mathtt{H^x}(o,x) =\bigl(a(x) X+c(x) Z, 0\big)$, where
$a,c$ are some functions of $x$.
Indeed, using relations~(\ref{eq.4.36}), (\ref{eq.4.7}), (\ref{eq.4.8}),
$[Z_0,Z]=0$ and the invariance of the form
$\langle\cdot,\cdot\rangle$, we obtain the following
expression at the point $(o,x)$
for any $\xi_1\in\fm\oplus\fk^+$, $t_1\in\bbR$:
\begin{equation}\label{eq.4.36x}
\begin{split}
\omega\Bigl(\bigl(\xi_1,t_1\tfrac{\partial}{\partial x}\bigr),&\,
\bigl(a X+c Z, 0\big)\Bigr)
= -\bigl\langle f'X+c_{Z}\bigl(\varphi-1\big)Z
-c_{Z}Z_0,[\xi_1,a X+c Z] \bigr\rangle \\
&+f''  t_1\bigl\langle  X, a X+c Z \bigr\rangle
+ c_{Z}\varphi'  t_1\bigl\langle Z,a X+c Z \bigr\rangle \\
&=(cf'-c_Z a\varphi)\bigl\langle Y,\xi_1 \bigr\rangle
+(af''+c_Zc\varphi')t_1,
\end{split}
\end{equation}
Now it is easy to see that
$\omega\bigl((\xi_1,t_1 \tfrac{\partial}{\partial x}),
\,\mathtt{H^x}\bigr)\eqdef
\rd x \bigl(\xi_1, \,t_1 \tfrac{\partial}{\partial x}\bigr)=t_1
$ at the point $(o,x)$ for arbitrary $t_1\in\bbR$,
$\xi_1\in\fm\oplus\fk^+$ if and only if
\begin{equation}\label{eq.4.39}
a=c\cdot\tfrac{f'}{c_Z\varphi}
\quad\text{and}\quad
\quad c=\tfrac{c_Z\varphi}{f''f'+c_Z^2\varphi'\varphi}.
\end{equation}
Since $J^K_c(\mathtt{H^x})(o,x) =\bigl(c\,\tfrac{\sinh x}{\cosh x}Y,
a \tfrac{\partial}{\partial x}\big)$
and $a=f_U^2$ we obtain at the point $(o,x)$
\begin{equation*}
\|\mathtt{H^x}\|^2\eqdef
\omega\bigl((c\tfrac{\sinh x}{\cosh x}Y,
a \tfrac{\partial}{\partial x}),(a X+c Z, 0)\bigr) =
\rd x \bigl(c\tfrac{\sinh x}{\cosh x}Y,
a \tfrac{\partial}{\partial x} \bigr)=a=f_U^2.
\end{equation*}
Now, by the Cauchy-Schwarz inequality for metrics one has
at the point $(o,x)$
\begin{equation*}
\begin{split}
\big|\rd x(\xi_1,\,t_1\tfrac{\partial}{\partial x})\big|
= & \;
\big|\omega((\xi_1,\,t_1 \tfrac{\partial}{\partial x}),
\mathtt{H^x})\big|
=\big|{\mathbf g}((\xi_1,\,t_1 \tfrac{\partial}{\partial x}),
J^K_c(\mathtt{H^x}))\big| \\ \noalign{\smallskip}
\; \leqslant & \;
\big\|J^K_c(\mathtt{H^x})\big\|\cdot \big\|(\xi_1,\, t_1
\tfrac{\partial}{\partial x})\big\|
=\|\mathtt{H^x}\|\cdot \big\|(\xi_1,\, t_1
\tfrac{\partial}{\partial x})\big\|,
\end{split}
\end{equation*}
that is, we obtain~\eqref{eq.4.35}. \qed

Using now the vector field $U$ we shall
calculate the distance between the level sets $G/H\times\{b\}$ and
$G/H\times\{c\}$ in $G/H\times \bbR^+$
with respect to the metric ${\mathbf g}$.
Let $\gamma(t)=(\widehat g(t)H$, $\widehat x(t))$, $t\in[0,T]$, be the
integral curve of the vector field $U$ with
initial point $p_b$ in $G/H\times\{b\}$, that is, $\widehat x(0)=b$.
There exists a function
$h$ on  ${\mathbb R}^+$ such that
the function $h(\widehat x(t))$ is linear in $t$.
It is easy to verify that
$\displaystyle h(x)=\int_{b}^{x}\frac{\rd s}{f_U(s)}\, $,
because by~(\ref{eq.4.34})
\begin{equation*}
\frac{\rd}{\rd t}h\bigl(\widehat x(t)\bigr)
=h'\bigl(\widehat x(t)\big)\cdot \rd x\bigl(\gamma'(t)\bigr)
=h'\bigl(\widehat x(t)\big)\cdot \widehat x'(t)
= h'\bigl(\widehat x(t)\big)\cdot
\bigl(f_U(\widehat x(t))\bigr)=1.
\end{equation*}
Suppose that $p_c\in G/H\times\{c\}$, where $p_c=\gamma(t_c)$,
$t_c\in[0, T]$.
Since the curve $\gamma$ is a geodesic, the length of the curve
$\gamma(t)$, $t\in[0,t_c]$, from
$p_b$ to $p_c$ is $t_c=h(x(p_c))-h \bigl(x(p_b)\big)=h(c)-h(b)$.
Thus $\mathrm{dist}(b,c)\geqslant h(c)-h(b)$.

For any other curve $\gamma_1(t)=\bigl(\widehat g_1(t)H,\widehat x_1(t)\big)$,
with $\|\gamma_1'(t)\|=1$, starting at the point $p_b$,  and ending
at a point $p^1_c\in G/H\times\{c\}$, $p^1_c= \gamma_1(t_c^1)$
(of length $t^1_c$), we obtain by Lemma~\ref{le.4.2}
\begin{equation*}
\frac{\rd}{\rd t}h\bigl(\widehat x_1(t)\bigr)
=h'\bigl(\widehat x_1(t)\big)\cdot \rd x
\bigl(\gamma_1'(t)\bigr)\leqslant \frac{1}{f_U\bigl(\widehat x_1(t)\big)}\cdot
f_U(\widehat x_1(t))\bigr)\cdot \|\gamma'_1(t)\|=1.
\end{equation*}

Thus $h(c)-h(b)\leqslant t^1_c$ and the length $t^1_c$ of
the curve $\gamma_1$  from $p_b$ to $p^1_c$ is not less than
the length of the curve $\gamma(t)$, $t\in[0,t_c]$.
Thus, the distance between the level surfaces $G/H\times\{b\}$ and
$G/H\times\{c\}$ is  $|h(c)-h(b)|$.

Now, since by~\eqref{eq.4.23} and~\eqref{eq.4.24} for $C_1=0$,
\begin{align*}
f'(x)&=\sqrt{C\sinh^2 x+c_Z^2\sinh^2 x\cosh^{-2}x}, \\
f''(x)&={\bigl(C\sinh x\cosh x
+c_Z^2 \sinh x\cosh^{-1}x
-c_Z^2 \sinh^{3} x\cosh^{-3}x\bigr)}/{f'(x)}, \notag
\end{align*}
we obtain that $f'(x)\sim \sqrt{C}\sinh x$,
$f''(x)\sim \sqrt{C}\sinh x$ and, by~(\ref{eq.4.33}),
$\frac{1}{f_U(x)}\sim\bigl(\sqrt{C}\sinh x\big)^{1/2} $ as
$x\to\infty$. Therefore
$\lim_{x\to\infty} h(x)=\infty$. Hence the metric
${\mathbf g}_0={\mathbf g}_0(C,c_Z, 0)$
(that is, for $C_1=0$)  on the tangent bundle
$T(G/K)$ is complete for any $C>0$, $c_Z\in\bbR$. \qed
\bigskip

It is well known that ${\mathbb R}{\mathbf P}^n\cong \mathbb{S}^n/\bbZ_2$
as ${\mathbb R}{\mathbf P}^n=\mathrm{SO}(n{+}1)/\mathrm{O}(n)$
and $\mathbb{S}^n=\mathrm{SO}(n{+}1)/\mathrm{SO}(n)$
$(n\geqslant 2)$. Hence each $\mathrm{SO}(n{+}1)$-invariant
Ricci-flat K\"ahler structure on $T{\mathbb R}{\mathbf P}^n$
is uniquely determined by a $\bbZ_2$-invariant
Ricci flat K\"ahler structure on $T{\mathbb S}^n$.

\begin{corollary}\label{co.4.3}
If $n\geqslant 3$, each
$G$-invariant Ricci-flat K\"ahler structure
$({\mathbf g}(C,C_1), J_c^K)$ on the punctured tangent
bundle
$T^+(G/K)=T^+(\mathrm{SO}(n{+}1)/\mathrm{SO}(n))=T^+{\mathbb S}^n$
determines an invariant Ricci-flat K\"ahler structure on
$T^+{\mathbb R}{\mathbf P}^n$. If $n=2$, the
$G$-invariant Ricci-flat K\"ahler structure
$({\mathbf g}(C,C_1, c_Z), J_c^K)$ on
$T^+(G/K)=T^+(\mathrm{SO}(3)/\mathrm{SO}(2))=T^+{\mathbb S}^2$
determines an invariant Ricci-flat K\"ahler structure on
$T^+{\mathbb R}{\mathbf P}^2$ if and only if
$c_Z=0$. All these invariant Ricci-flat K\"ahler metrics on
$T^+{\mathbb R}{\mathbf P}^n$ are uniquely extendable to
complete metrics on the whole tangent bundle
$T{\mathbb R}{\mathbf P}^n$,
$n\geqslant 2$, if and only if $C_1=0$.
\end{corollary}
\begin{proof}
We will use the notations of the proof of Theorem~\ref{th.4.1}.
As it follows from its proof the
K\"ahler structure $({\mathbf g}(C,C_1), J_c^K)$
on $T^+(G/K)=T^+(\mathrm{SO}(n{+}1)/\mathrm{SO}(n))$
$(n\geqslant 3)$ is $\bbZ_2$-invariant if and only if the form
$\tio^R=\Delta$ (see~(\ref{eq.4.27})) on $G\times\fm$
is right $K_1$-invariant, where $K_1=\mathrm{O}(n)$
($K\subset K_1\subset G$). The form $\tio^R$ is
right $K_1$-invariant because $\Ad(K_1)(\fm)=\fm$
and $\Ad(K_1)$ is a subgroup of the group of inner
automorphisms $\Ad(G)$ of $\fg$. Similarly,
if $n=2$, the form $\tio^R=\Delta$ (see~(\ref{eq.4.31})) is
right $K_1$-invariant if and only if
$c_Z=0$ because $\Ad(K_1)Z\ne \{Z\}$ ($Z_0=Z$ if $n=2$
and ${\mathbb R}{\mathbf P}^2$ is not a homogeneous
complex manifold). Now the last assertion
of the corollary immediately follows from the last assertion
of Theorem~\ref{th.4.1}.
\end{proof}

\end{document}